\documentclass[11pt]{amsart}
 
\usepackage{mathtools}
\mathtoolsset{showonlyrefs,showmanualtags}
 \usepackage{hyphenat}
\usepackage{amssymb, amsmath, amsthm, esint}
\usepackage{mathrsfs}
\usepackage{color}
\usepackage{comment}
\usepackage{graphicx}

\usepackage[top=1.5in, bottom=1.5in, left=1.25in, right=1.25in]{geometry}
\usepackage[all]{xy}
\usepackage{amscd}
 \usepackage{microtype}
\usepackage[bottom]{footmisc}
 
\usepackage{fancyhdr}
 
\AtBeginDocument{\setlength{\footskip}{24pt}}
\fancypagestyle{plain}{%
  \fancyhf{}%
  \fancyfoot[C]{\thepage}%
}
 
\fancypagestyle{myfancy}{%
  \fancyhf{}%
  \fancyhead[CE]{\scshape Pointwise Convergence of Random Ergodic Averages}%
  \fancyhead[CO]{\scshape Burstein--Catani--Krause}%
  \fancyhead[LE,RO]{\thepage}%
}
 
\numberwithin{equation}{section}

\newtheorem{theorem}{Theorem}[section]
\newtheorem*{theorem*}{Theorem}

\newtheorem{proposition}[theorem]{Proposition}
\newtheorem*{proposition*}{Proposition}
\newtheorem{cor}[theorem]{Corollary}
\newtheorem{lemma}[theorem]{Lemma}
\newtheorem*{lemma*}{Lemma}
\newtheorem{remark}[theorem]{Remark}

\theoremstyle{definition}

\usepackage[
  pagebackref,
  colorlinks=true,
  linkcolor=blue,
  citecolor=magenta,
  filecolor=magenta,
  urlcolor=cyan
]{hyperref}

\makeatletter
\renewcommand*{\backref}[1]{}
\renewcommand*{\backrefalt}[4]{%
  \ifcase #1\relax
  \or
    \space (Page~#2)%
  \else
    \space (Pages~#2)%
  \fi
}
\makeatother

\begin{document}
 
\title{Pointwise Convergence of Random Ergodic Averages at the \(L^1\) Endpoint}
 
\author{Will Burstein}
\address{%
Department of Mathematics, University of Rochester \\
14627, Rochester, NY, Hylan Building, 915}
\email{willburst88@gmail.com}

\author{Lorenzo Catani}
\address{%
Department of Mathematics, Princeton University \\
Fine Hall, 304 Washington Rd, Princeton, NJ 08540, United States}
\email{lc9651@princeton.edu}

\author{Ben Krause}
\address{%
  Department of Mathematics, University of Bristol \\
  Beacon House, Queens Rd, Bristol BS8 1QU}
\email{ben.krause@bristol.ac.uk}

\date{September 2026}
 
\begin{abstract}
In his highly influential paper, \emph{On the maximal ergodic theorem for certain subsets of the integers}, Bourgain introduced the study of pointwise convergence of ergodic averages along randomly generated subsets of the integers, developing a robust \(L^p\)-theory for \(p>1\). More recently, this theory was partially complemented at the \(L^1\) endpoint by work of LaVictoire, who treated averages along random sequences that are ``slightly denser" than the set of squares, for which universal \(L^1\) pointwise convergence is known to fail. In this work, we prove that, almost surely, for every measure-preserving system and every integrable function, the ergodic averages formed by sampling along random sequences with quadratic growth converge almost everywhere. Our proof combines harmonic-analytic methods developed by M.~Christ with a ``random sampling" of quadratic forms, thereby overcoming the main uniformity obstruction.
\end{abstract}
 
\maketitle
 
\setcounter{footnote}{0} % Reset footnote counter after \maketitle
 
\pagestyle{myfancy}
 
\setcounter{tocdepth}{1}
\tableofcontents

\section{Introduction}

Let \((X,\mathcal F,\mu)\) be a \(\sigma\)-finite measure space and let
\(T:X\to X\) be a measure-preserving transformation, that is, a measurable map such that
\[
    \mu(T^{-1}(E))=\mu(E),
    \qquad E\in\mathcal F.
\]
We refer to the quadruple \((X,\mathcal F,\mu,T)\) as a \emph{measure-preserving system}. The field of pointwise ergodic theory is concerned with the pointwise limiting behavior, as \(N\to\infty\), of averages of the form
\begin{equation}\label{eq: averages}
\frac{1}{N}\sum_{n\leq N}T^{a_n}f,
\end{equation}
where \(\{a_n\}_{n\in\mathbb N}\subseteq\mathbb N\), \(f\) belongs to an appropriate function space on \(X\), and
\[
    (T^n f)(x):=f(T^n x),
    \qquad x\in X;
\]
the pointwise ergodic theorem of Birkhoff \cite{Birkhoff1931} asserts that, when \(a_n=n\), the averages in \eqref{eq: averages} converge \(\mu\)-almost everywhere for every \(f\in L^1(X,\mu)\).

A classical question in the subject, which remained open for nearly sixty years, concerns the pointwise limiting behavior of the averages in \eqref{eq: averages} when they are taken along sparse subsets of the integers. More specifically, one asks whether there exists a sequence \(\{a_n\}\subseteq\mathbb N\) of zero upper Banach density, that is,
\[
    d^*(\{a_n\})
    :=
    \limsup_{L\to\infty}
    \sup_{\substack{I\subseteq\mathbb N \, \operatorname{an \, interval}\\ |I|=L}}
    \frac{|\{a_n\}\cap I|}{|I|}
    =0,
\]
for which the averages in \eqref{eq: averages} nevertheless converge \(\mu\)-almost everywhere.

In a landmark series of papers \cite{maximalergodic,ergodicLp,B1} published in the late \(1980\)s, Bourgain answered this question in the affirmative for the sequence of squares
\[
    a_n=n^2,
\]
and, more generally, for integer-valued polynomial sequences. His argument introduced a novel combination of Fourier-analytic, number-theoretic and probabilistic techniques. For the purpose of this paper, we record the case of the squares below. 

\begin{theorem}[Polynomial Ergodic Theorem]\label{thm: poly ergodic}
Let \((X,\mathcal F,\mu,T)\) be a measure-preserving system, and let \(p>1\). Then, for every \(f\in L^p(X,\mu)\), the averages
\[
    \frac{1}{N}\sum_{n\leq N}T^{n^2}f
\]
converge pointwise \(\mu\)-almost everywhere as \(N\to\infty\).
\end{theorem}

In \cite{maximalergodic}, Bourgain also initiated the study of ergodic averages taken along randomly generated sequences of integers. The setup is as follows. Let \(\{X_n\}_{n\in\mathbb N}\) be a collection of independent Bernoulli random variables defined on a common probability space \((\mathbf{\Omega},\mathcal A,\mathbb P)\), with
\[
    \mathbb E(X_n)=\sigma_n\in(0,1],
\]
where \(\{\sigma_n\}\) is a decreasing sequence and $\sum_{n \in \mathbb N} \sigma_n = \infty$. The corresponding random sequence of \emph{hitting times} is defined by
\[
    a_n
    :=
    \min\Big\{
        k\in\mathbb N:
        X_1+\cdots+X_k=n
    \Big\};
\]
equivalently, \(\{a_n\}_{n \in \mathbb N}\) is the increasing enumeration of the random set
\[
    \big\{n\in\mathbb N:X_n=1\big\}.
\]
The decay of the selector probabilities \(\sigma_n\) determines the sparsity of the resulting sequence, in that
\begin{align}
    \frac{1}{m} \sum_{k \leq a_m} \sigma_k \to 1
\end{align}
by a brief argument involving e.g.\ Bernstein's inequality, see Theorem \ref{thm: standard bernstein} below.

Bourgain proved that this random model exhibits remarkably strong pointwise convergence properties away from the \(L^1\) endpoint. In particular, for every \(p>1\), he exhibited selector probabilities \(\{\sigma_n\}\) tending to zero for which, almost surely, the corresponding ergodic averages converge pointwise \(\mu\)-almost everywhere for every \(f\in L^p(X,\mu)\) and every $\sigma$-finite measure-preserving system. For example, when \(1<p\leq2\), Bourgain showed that the choice
\[
    \sigma_n
    =
    \frac{(\log\log n)^{B_p}}{n},
    \qquad
    B_p>\frac{1}{p-1},
\]
yields almost-sure pointwise convergence of the associated ergodic averages for every \(f\in L^p\).

The endpoint \(p=1\), however, is considerably more subtle. A striking example is provided by the squares themselves: whereas Theorem~\ref{thm: poly ergodic} yields pointwise convergence along \(n^2\) for every \(f\in L^p\), \(p>1\), Buczolich and Mauldin \cite{L1bad} showed that this conclusion need not persist at the \(L^1\) endpoint. On the other hand, Buczolich \cite{Buc} exhibited sequences of zero-upper Banach density for which the pertaining ergodic averages converged for every $f \in L^1$; subsequently, Urban and Zienkiewicz \cite{UZ} proved that for every
\[
    1<c<\frac{1001}{1000},
\]
the ergodic averages taken along the sequence
\[
    a_n=\lfloor n^c\rfloor
\]
converge pointwise for every $f\in L^1$.
This result was strengthened by Mirek \cite{M}, with the state of the art result due to the third author and Sun \cite{KS}, which extended the range to $c < 7/6$.

For random sequences, a positive \(L^1\) result was obtained by LaVictoire \cite{victoire}: if 
\[
    \sigma_n=n^{-\alpha},
    \qquad
    0<\alpha<\frac12,
\]
then, almost surely, the corresponding ergodic averages converge pointwise \(\mu\)-almost everywhere for every \(f\in L^1(X,\mu)\) and every $\sigma$-finite measure-preserving system. In this case,
$$
    \sum_{n\leq N}\sigma_n\approx N^{1-\alpha},
$$
where the notation \(\approx\) is defined in Section~\ref{sec: notation}.
As alluded to above, by Bernstein's inequality, this implies that, almost surely, there exists a finite random constant \(C_\omega>0\) such that the \(n\)-th hitting time satisfies
\begin{equation}\label{eq: growth rate random seq}
C_\omega^{-1}n^{\frac{1}{1-\alpha}}
\leq
a_n
\leq
C_\omega n^{\frac{1}{1-\alpha}}.
\end{equation}
In particular, LaVictoire's result produces random sequences with growth arbitrarily close to quadratic for which pointwise convergence holds for every \(L^1\) function; it does \emph{not} reach the critical value
$
    \alpha=1/2, 
$
at which the hitting times exhibit quadratic growth.

The quadratic growth condition is remarkably subtle: convergence in this case is not determined by the growth rate alone. Indeed, in \cite{Christ2011WeakTypeSparse}, Christ constructed sequences satisfying
\[
    a_n\approx n^2
\]
for which the corresponding ergodic averages converge pointwise \(\mu\)-almost everywhere for every \(f\in L^1(X,\mu)\) and every $\sigma$-finite measure-preserving system, as part of a more general phenomenon. His argument relies on Fourier-analytic conditions which ensure that the sparse averaging measures are sufficiently well approximated by normalized indicators of intervals. Thus, even among sequences with quadratic growth, the pointwise \(L^1\) behavior can differ dramatically from that of the squares.

In this paper, we focus on the critical case
\[
    \sigma_n=n^{-1/2},
\]
which corresponds, almost surely, to random sequences with quadratic growth. We prove that, with probability one, the associated ergodic averages converge pointwise \(\mu\)-almost everywhere for every integrable function. Specifically, we establish the following theorem. 

\begin{theorem}\label{thm: main theorem}
Let \(\{X_n\}_{n\in\mathbb N}\) be independent Bernoulli random variables defined on a common probability space \((\mathbf{\Omega},\mathcal A,\mathbb P)\), with
\[
    \mathbb E(X_n)=n^{-1/2}.
\]
For \(n\in\mathbb N\), let
\[
    a_n
    :=
    \min\Big\{
        k\in\mathbb N:
        X_1 + \dots +  X_k=n
    \Big\}
\]
denote the corresponding random sequence of hitting times. Then, with probability one, the following holds simultaneously for every \(\sigma\)-finite measure-preserving system \((X,\mathcal F,\mu,T)\): for every
\[
    f \in L^1(X, \mu)
\]
the ergodic averages
\begin{equation}\label{eq: random means}
    \frac1N\sum_{n\leq N}T^{a_n}f
\end{equation}
converge pointwise \(\mu\)-almost everywhere.
\end{theorem}

In light of Bourgain's $L^2$ work, Theorem \ref{thm: main theorem} follows from an appropriate weak-type maximal inequality for finitely supported functions on \(\mathbb Z\), stated below as Theorem~\ref{t:main}. This implication follows from a standard transference and approximation argument which we review for completeness in Section~\ref{sec: reduction to max inequality}.

The proof of the maximal inequality can be conceptually divided into two parts. In the first part, we carry out a sequence of harmonic-analytic reductions inspired by \cite{Christ2011WeakTypeSparse}. These include a Calder\'on--Zygmund localization, a decomposition into dyadic level sets and a decomposition of the relevant scales into a \emph{short-gap} and a \emph{long-gap} contribution. After these reductions, the main difficulty is reduced to estimating the \(\ell^2\) norm of convolutions of a centered random kernel with indicators of sets of prescribed cardinality. The key estimate is recorded in Proposition \ref{p:key}.

The second part is devoted to proving this restricted estimate. We begin with a symmetrization argument that reduces the problem to a finite model in which sets of controlled cardinality are convolved with a random signed kernel supported on a fixed sparse set. The main difficulty is then one of uniformity: the class of admissible sets is too large to permit a direct probabilistic union bound. To overcome this, we introduce an auxiliary sampling procedure which replaces each set \(A\) by a small random sample of points from \(A\). With high probability, the sampled objects retain the analytic properties needed in the proof, while belonging to a family of much smaller complexity. This reduction makes it possible to obtain the required uniform estimate by combining matrix concentration inequalities with a simple union bound, and then to transfer the resulting control back to the original deterministic family. For a detailed overview of the proof of the fixed-support estimate, stated below as Proposition~\ref{prop: fixed support prop}, see Section~\ref{sec: proof overview}.

We mention that an earlier argument of the second author gave a weaker
form of Proposition~\ref{p:key}. Combined with a harmonic-analytic argument very similar to the one described below, this yields the analog of Theorem \ref{thm: main theorem} for functions which satisfy the stronger integrability condition $$\int_X |f(x)|\log \log \log (10^{10}+|f(x)|)\, d\mu(x) < \infty.$$
That argument makes use of Fourier analysis and metric-entropy estimates; a high-level sketch is given in Appendix~\ref{app:old-argument}.

We expect that the method developed in this paper is not restricted to the critical exponent \(\alpha=1/2\). In particular, the parameter balances arising in the fixed-support argument suggest that the same general strategy should extend to selector probabilities
\[
\sigma_n=n^{-\alpha}
\]
for at least some range of exponents \(\alpha>1/2\). Preliminary considerations suggest that the argument may extend up to \(\alpha<2/3\); since our main objective is the critical case \(\alpha=1/2\), we do not pursue the optimization of this range here.
\subsection*{Acknowledgments} 
The third author benefited from conversations with Michael Christ and Terence Tao during an MSRI Program in Harmonic Analysis, held in Spring 2017.

The development of this approach benefited in part from mathematical suggestions generated in discussions with ChatGPT.  The authors take responsibility for the mathematical statements and proofs.

\subsection{Remarks on Notation}\label{sec: notation}
We employ modified Vinogradov notation throughout, writing $X \lesssim Y$, 
or equivalently $Y \gtrsim X$, to denote the bound $X \leq CY$ for some absolute constant $C > 0$. When the constant depends on a parameter, we indicate this by a subscript: for instance, $X \lesssim_p Y$ means $X \leq C_p Y$ for some 
constant $C_p$ depending on $p$. In particular, we will write
\[
A\lesssim_\omega B
\]
to indicate that the estimate holds almost surely, with an implicit constant which may depend on the underlying random outcome \(\omega\). 
We write $X \approx Y$ as shorthand for 
$Y \lesssim X \lesssim Y$, and analogously define $X \approx_p Y$.

We also use the standard big-$O$ notation: $O(Y)$ denotes a quantity whose 
magnitude is $\lesssim Y$, and $O_p(Y)$ a quantity whose magnitude is 
$\lesssim_p Y$.

We denote by $e(t) := e^{2\pi i t}$ the standard additive character on $\mathbb{T} \coloneqq \mathbb{R}/\mathbb{Z} \cong [0,1)$. For a finitely supported function $f : \mathbb{Z} \to \mathbb{C}$, we define its Fourier transform $\widehat{f} : \mathbb{T} \to \mathbb{C}$ by
\[
    \widehat{f}(\beta) \coloneqq \sum_{n \in \mathbb{Z}} f(n)\, e(-n\beta), \qquad \beta \in \mathbb{T},
\]
with inversion formula
\[
    f(n) = \int_{\mathbb{T}} \widehat{f}(\beta)\, e(n\beta)\, d\beta, \qquad n \in \mathbb{Z},
\]
and Plancherel's identity
\[
    \sum_{n \in \mathbb Z} |f(n)|^2 = \int_{\mathbb{T}} |\widehat{f}(\beta)|^2\, d\beta.
\]

We write $M_{\mathrm{HL}}$ for the Hardy--Littlewood maximal operator, which in the discrete setting takes the form 
\[
    M_{\mathrm{HL}}f(x) \coloneqq \sup_{N \geq 1} \frac{1}{2N+1} \sum_{n = -N}^N |f(x-n)|,
\]
and its dyadic variant
\begin{equation}\label{eq: dyadic HL}
 M_\Delta f(x) := \sup_{Q \ni x: \,Q \in \mathcal D} \frac{1}{|Q|} \sum_{y \in Q}|f(y)|, 
 \end{equation}
where $$\mathcal D := \bigcup_{n \in \mathbb N \cup \{0\}} \mathcal D (n), \qquad \mathcal D(n) := \Big\{ \big[2^n m,2^n (m+1) \big)\cap\mathbb Z:
m\in\mathbb Z  \Big\}
$$ are the sets of \emph{dyadic intervals}.

For a real matrix $B$, we write $B^{\mathsf T}$ for its transpose,
$\|B\|_F^2=\operatorname{tr}(B^{\mathsf T}B)$ for its squared
Frobenius norm, and $\|B\|_{\mathrm{op}}$ for its Euclidean operator
norm. For symmetric matrices, $B\preceq C$ means that $C-B$ is positive
semidefinite. We use $\varepsilon$ for a vector of independent Rademacher signs,
with law $\mathbb P_\varepsilon$, while $\mathbb P_{\mathrm{aux}}$ and
$\mathbb E_{\mathrm{aux}}$ refer to the auxiliary sampling introduced
in Section~\ref{sec:sampling}.

\section{Problem Setup and Reduction to a Maximal Inequality}
\label{sec: reduction to max inequality}

Throughout the paper, we let \(\{X_n\}_{n\in\mathbb N}\) be a collection of independent Bernoulli random variables defined on a common probability space
    $(\mathbf{\Omega},\mathcal A,\mathbb P)$,
with expectations
\[
    \mathbb E(X_n)=\sigma_n=n^{-1/2}.
\]

We begin by reformulating the random ergodic averages in terms of the selector variables \(X_n\). Recall that \(\{a_n\}_{n\in\mathbb N}\) denotes the increasing enumeration of the random set
\[
    \{n\in\mathbb N:X_n=1\}.
\]
It follows that convergence of the random ergodic means in \eqref{eq: random means} is equivalent to convergence of the averages
\begin{equation}\label{eq: averages with X_n}
\frac{1}{\sum_{n\leq N}X_n}
\sum_{n\leq N}X_n \cdot T^nf
\end{equation}
as \(N\to\infty\).
Since
\[
    \sum_{n\leq N}\sigma_n\rightarrow\infty,
\]
the selector count satisfies, with probability one,
\begin{equation}\label{eq: ssn approx}
\frac{\sum_{n\leq N}X_n}
{\sum_{n\leq N}\sigma_n}
\longrightarrow1.
\end{equation}
This can be proven quickly by way of Bernstein's inequality, which will be used repeatedly throughout the paper.
\begin{theorem}[Bernstein's inequality]\label{thm: standard bernstein}
Let \(Z_1,\dots,Z_N\) be independent, mean-zero random variables satisfying
\[
    |Z_n|\leq1
\]
almost surely for every \(n\). Set
\[
    \sigma^2
    :=
    \sum_{n=1}^N\mathbb E(|Z_n|^2).
\]
Then, for every \(t>0\),
\[
    \mathbb P\Big(
        \Big|
            \sum_{n=1}^N Z_n
        \Big|
        \geq t
    \Big)
    \lesssim
    \exp\Big(
        -c\min\Big\{
            \frac{t^2}{\sigma^2},
            t
        \Big\}
    \Big),
\]
where \(c>0\) is an absolute constant.
\end{theorem}

\begin{proof}
See, for instance, \cite[Theorem 1.8]{Tao_Vu}.
\end{proof}

In view of \eqref{eq: ssn approx}, we may replace the averages in
\eqref{eq: averages with X_n} by the asymptotically equivalent averages
\begin{equation}\label{eq: averages with sigma_n}
A_N^{(\omega)}f
:=
\frac{1}{\sum_{n\leq N}\sigma_n}
\sum_{n\leq N}X_n \cdot T^nf.
\end{equation}
We next reduce Theorem~\ref{thm: main theorem} to a maximal estimate for this family of operators.

\begin{theorem}\label{thm: general max inequality}
Almost surely, there exists a finite random constant \(C_\omega>0\) such that the following holds simultaneously for every \(\sigma\)-finite measure-preserving system \((X,\mathcal F,\mu,T)\), every $f \in L^1(X, \mu)$
and every \(\lambda>0\):
\[
    \mu\Big(
        \Big\{
            x\in X:
            \sup_{N\in\mathbb N}
            |A_N^{(\omega)}f(x)|
            \geq
            \lambda
        \Big\}
    \Big)
    \leq
    \frac{C_\omega}{\lambda} \|f\|_{L^1(X,\mu)}
\]

\end{theorem}

\begin{proof}[Proof of Theorem~\ref{thm: main theorem} assuming Theorem~\ref{thm: general max inequality}]
The following is a standard approximation argument. 
Fix \(\omega\) in the probability-one event on which the maximal inequality holds. By density, every $f \in L^1(X, \mu)$
may be approximated by functions \(g_j\in L^2(X,\mu)\) so that $$\|f-g_j\|_{L^1(X, \mu)} \rightarrow 0$$
as \(j\to\infty\).
By Bourgain's theorem, \(A_N^{(\omega)}g_j\) converges pointwise \(\mu\)-a.e.\ for each $j$, after possibly restricting to a smaller universal probability--one event. Hence
\[
    \limsup_{M,N\to\infty}
    |A_N^{(\omega)}f-A_M^{(\omega)}f|
    \leq
    2\sup_{N\in\mathbb N}
    |A_N^{(\omega)}(f-g_j)|
\]
for \(\mu\)-a.e. \(x \in X\). Therefore, Theorem~\ref{thm: general max inequality} yields the bound
\[
    \mu\Big(
        \Big\{
            x\in X:
            \limsup_{M,N\to\infty}
            |A_N^{(\omega)}f(x)-A_M^{(\omega)}f(x)|
            \geq
            t
        \Big\}
    \Big)
    \leq
    C_\omega
    \frac{\|f-g_j\|_{L^1(X, \mu)}}{t}, \qquad t > 0.
\]
First sending \(j\to\infty\) and then \(t\to0\), we conclude that
$
    \{A_N^{(\omega)}f(x)\}_{N\in\mathbb N}
$
is Cauchy for \(\mu\)-a.e. \(x\in X\), and hence converges. In view of \eqref{eq: ssn approx}, the original random ergodic means converge as well. This proves Theorem~\ref{thm: main theorem}.
\end{proof}
It remains to establish Theorem~\ref{thm: general max inequality}. By Calderón's Transference Principle \cite{calderon}, it is enough to prove the corresponding quantitative estimate for finitely supported functions on \(\mathbb Z\), with the transformation \(T\) replaced by the shift. 

Note that, by positivity of the averaging operators, it suffices to control the dyadic subsequence. Moreover, each dyadic initial segment may be decomposed into dyadic annuli, and positivity allows us to dominate the resulting averages by averages over the ``upper halves" of these intervals. 

Consequently, for $k \in \mathbb N$ we define the dyadic interval
\[
    I_k
    :=
    \big(2^{k-1},2^k\big]\cap\mathbb Z
\]
and its expected selector count
\begin{equation}\label{eq: W_k approx}
    W_k
    :=
    \sum_{n\in I_k}\sigma_n \approx 2^{k/2}.
\end{equation}
Define the centered random variables
\[
    Y_n:=X_n-\sigma_n
\]
and the random measures

\[
    \mu_k
    :=
    \frac{1}{W_k}
    \sum_{n\in I_k}X_n\delta_n,
    \qquad
    \nu_k
    :=
    \mu_k-\mathbb E(\mu_k)
    =
    \frac{1}{W_k}
    \sum_{n\in I_k}Y_n\delta_n.
\]
The relevant maximal operator is then given by
\[
    Mf
    :=
    \sup_{k\geq1}
    |\mu_k*f|,
    \qquad
    f:\mathbb Z\to\mathbb C;
\]
it suffices to establish the following discrete maximal estimate.
\begin{theorem}\label{t:main}
Almost surely, there exists a finite random constant \(C_\omega > 0\) such that, for every finitely supported function
$
    f:\mathbb Z\to\mathbb C
$
and every \(\lambda>0\),
\[
    \Big|
        \Big\{
            x\in\mathbb Z:
            Mf(x)\geq\lambda
        \Big\}
    \Big|
    \leq
    \frac{C_{\omega}}{\lambda} 
    \sum_{x \in \mathbb Z} |f(x)|.
\]
\end{theorem}
The remainder of the paper is devoted to proving Theorem~\ref{t:main}; we begin with several preliminary reductions inspired by the argument in \cite{Christ2011WeakTypeSparse}. 

Notice that, since each of the kernels $\mu_k$ is nonnegative,
$$\sup_{k\geq1}
    |\mu_k*f| \leq \sup_{k\geq1}
    \mu_k*|f|. $$Thus there is no loss of generality in assuming that $f \geq 0,$ and we do so henceforth.
\section{Preliminary Reductions}\label{sec:preliminary-reductions}
\subsection{Two standard estimates}
We begin by establishing two bounds that will be useful in what follows. The first is an almost-sure estimate on the $L^\infty(\mathbb{T})$ norm of the Fourier transform of the centered random measures $\nu_k$. We record the estimate in the following proposition.
\begin{proposition}\label{thm: fourier bound nu_k}
   With the notations established in the introduction, almost surely there exists a finite random constant $C_\omega$ such that
   \begin{equation}\label{eq: fourier transform estimate nu}
       \|\widehat{\nu_k} \|_{L^\infty(\mathbb{T})} \le C_\omega \sqrt{\frac{k}{W_k}}
   \end{equation}
   for all $ k \ge 1.$
\end{proposition}
\begin{proof}
    The proof is standard. By Bernstein's inequality (Theorem~\ref{thm: standard bernstein}) we obtain the pointwise bound $$\sup_{\beta \in \mathbb{T}} \mathbb{P}\big(|\widehat{\nu_k }(\beta)| \geq A \cdot k^{1/2} \cdot W_k^{-1/2}\big) \lesssim \operatorname{exp}(-cA^2k),$$ where $c >0$ is an absolute constant and $A > 0$ is a (sufficiently large) fixed constant. Exploiting the smoothness of trigonometric polynomials of fixed degree yields the estimate $$\|\widehat{\nu_k}\|_{L^\infty(\mathbb{T})} \lesssim \max_{\beta \in \Delta_k} |\widehat{\nu_k}(\beta)|,$$ where $$\Delta_k := \Big\{ \frac{j}{2^{k}A}: 1 \leq j \leq 2^k A \Big\} \subset \mathbb{T}.$$ A union bound over the frequency points in $\Delta_k$ then yields $$\mathbb{P}(\|\widehat{\nu_k}\|_{L^\infty(\mathbb{T})} \geq A \cdot k^{1/2} \cdot W_k^{-1/2}) \lesssim 2^k \cdot \operatorname{exp}(-cA^2k) \lesssim k^{-10}$$ if $A$ is sufficiently large. The result then follows by Borel--Cantelli. 
\end{proof}
The second result involves the deterministic approximation $\mathbb{E}(\mu_k)$. We record it below as a lemma; the proof is immediate. 
\begin{lemma}\label{lem: deterministic bound}
    Given a function $f : \mathbb{Z} \rightarrow \mathbb{C}$, $$|\mathbb{E}(\mu_k) \ast f(x)| \lesssim M_{\operatorname{HL}}f(x)$$ uniformly in $x \in \mathbb{Z}$ and $k \geq 1.$
\end{lemma}
\subsection{The high-value exceptional set}

For $k \geq 1$, consider the random set
\[
    \operatorname{supp} \mu_k 
    =\{n\in I_k:X_n=1\}.
\]
By \eqref{eq: ssn approx}, almost surely
there exists a finite random constant \(C_\omega > 0\) such that
\begin{equation}\label{eq: strong lln}
    |\operatorname{supp} \mu_k |\leq C_\omega W_k
\end{equation}
for every \(k\geq 1\), after enlarging \(C_\omega\) to absorb the
finitely many exceptional scales.

Fixing a sufficiently small constant \(c>0\), we define the high-value exceptional set
\begin{equation}\label{eq: exceptional set}
    X_{\text{high}}
    :=
    \bigcup_{k\geq 1}
    \Big(
        \operatorname{supp} \mu_k +\{x\in\mathbb Z:f(x)\geq cW_k\}
    \Big).
\end{equation}
By \eqref{eq: strong lln} and \eqref{eq: W_k approx}, almost surely
\begin{align}
    |X_{\text{high}}| &\leq
    \sum_{k\geq 1}
    |\operatorname{supp} \mu_k |\,
    \big|\{x \in \mathbb Z:f(x)\geq cW_k\}\big| \\
    &\lesssim_\omega \sum_{k \geq 1}\sum_{x \in \mathbb Z} W_k \cdot \mathbf{1}_{\{f(x) \geq c W_k\}}\\  
    & = \sum_{x \in \mathbb Z } \sum_{k: \, cW_k \leq f(x)}W_k \\
    & \lesssim \|f\|_{\ell^1}.
\end{align}
Note that for each $k$, $|f(x-n)| \lesssim W_k$ 
whenever $ x \notin X_{\text{high}}$ and $n \in \operatorname{supp} \mu_k $.

For each \(m\in\mathbb Z\), define
\[
    E_m:=\{x\in\mathbb Z:2^m\leq f(x)<2^{m+1}\},
    \qquad
    f^m:=2^m\mathbf 1_{E_m}.
\]
Since the kernels \(\mu_k\) are nonnegative, up to
an absolute constant we may replace \(f\) with its discretized dyadic version
\[
    \sum_{m\in\mathbb Z}f^m;
\]
we shall make this reduction implicitly below.

With this notation, the preceding high-value restriction implies that only
levels satisfying
\[
    m\leq \ell_k+O(1),
    \qquad
    \ell_k:=\Big\lfloor\frac{k}{2}\Big\rfloor,
\]
can contribute at scale \(k\). Since \(W_k\lesssim 2^{\ell_k}\), we may choose the constant \(c>0\) in
\eqref{eq: exceptional set} sufficiently small so that
\[
    cW_k\leq 2^{\ell_k}
\]
for every \(k\geq1\). Consequently, outside \(X_{\mathrm{high}}\), if the
dyadic level \(E_m\) contributes at scale \(k\) then \(m\leq \ell_k\), and we may parameterize all such levels as
\[
    m=\ell_k-s,
    \qquad 0 \leq s\leq \ell_k;
\]
below, we implicitly restrict $s$ to this range.
\subsection{Calder\'on--Zygmund intervals and block localization}
With $M_{\Delta}$ as in \eqref{eq: dyadic HL}, let 
\[
    E:=\{x \in \mathbb{Z}:M_{\Delta}f(x)\geq 1\}
    =\bigcup_{Q\in\mathcal I}Q,
\]
where \(\mathcal I \subseteq \mathcal D\) is the disjoint family of maximal dyadic intervals
contained in $E$. If we set $g := f- f\cdot \mathbf{1}_E$, then by construction $$\|g\|_{\ell^\infty} \le 1, \qquad \|g\|_{\ell^1} \leq \|f\|_{\ell^1}.$$ Making use of Proposition \ref{thm: fourier bound nu_k} and Lemma \ref{lem: deterministic bound} yields the almost-sure $\ell^2(\mathbb Z)$ estimate
\begin{align}\label{eq: l^2 estimate}
    \|\sup_{k \geq 1} |\mu_k \ast g| \|_{\ell^2}^2 &\lesssim \sum_{k \geq 1} \|\nu_k \ast g\|_{\ell^2}^2 + \|\sup_{k \geq 1} |\mathbb{E}(\mu_k) \ast g|\|_{\ell^2}^2 \\
    & \lesssim \sum_{k \geq 1} \|\widehat{\nu_k}\|_{L^\infty(\mathbb T)}^2 \|g\|_{\ell^2}^2 + \|M_{\operatorname{HL}}g\|_{\ell^2}^2 \notag \\
    & \lesssim_\omega \|g\|_{\ell^2}^2\sum_{k \geq 1} k \cdot W_k^{-1} + \|g\|_{\ell^2}^2 \notag \\
    & \lesssim \|g\|_{\ell^2}^2 \leq \|g\|_{\ell^\infty} \|g\|_{\ell^1}\lesssim \|f\|_{\ell^1}; \notag
\end{align}
we may therefore replace $f$ with $f \cdot \mathbf{1}_E$, and we do so henceforth. 

If we next set
\[
    E_{\text{CZ}}
    :=\bigcup_{Q\in\mathcal I}10Q,
\]
where $10Q$ indicates the interval concentric with $Q$ with 10 times the length, then 
\[
    |E_{\text{CZ}}|
    \lesssim
    \|f\|_{\ell^1}
\]
by the (dyadic) Hardy-Littlewood inequality. Moreover, the maximality of each \(Q\in \mathcal I\) implies that
\[
    \sum_{x\in Q}f(x)
    \lesssim |Q|.
\]

For \(\kappa\in\{0,1/3,2/3\}\), we define the collection of sets
\[
    \mathcal D^\kappa
    :=
    \Big\{
        2^r\big([m,m+1)+\kappa(-1)^r \big):
        m\in\mathbb Z,
        \ r\in\mathbb Z
    \Big\}.
\]
If \(J\) is an interval, we write \(J/3\) for its central third, which we take to be half-open. With this convention in mind, for each
fixed \(r \in \mathbb{Z}\) the central thirds satisfy the partition identity
\[
    \sum_{\kappa\in\{0,1/3,2/3\}}
    \sum_{\substack{J\in\mathcal D^\kappa\\|J|=2^r}}
    \mathbf 1_{J/3}(x)
    =1,
    \qquad x\in\mathbb R.
\]
By arguing in physical space, we next observe that whenever \(x\notin E_{\text{CZ}}\),
\(n\in I_k\) and \(x-n\in Q\in\mathcal I\), the length of $Q$ is necessarily constrained, namely
\[
    |Q|\lesssim 2^k.
\]
In particular, if \(C_{\mathrm{loc}}\) is a sufficiently large absolute constant and
\[
    \mathcal I_k^{\mathrm{sm}}
    :=
    \Big\{
        Q\in\mathcal I:
        |Q|\leq C_{\mathrm{loc}}2^k
    \Big\},
\]
then away from \(E_{\text{CZ}}\) only intervals in
\(\mathcal I_k^{\mathrm{sm}}\) can contribute to the convolution at
scale \(k\).

Consequently, for \(\kappa\in\{0,1/3,2/3\}\) and
\(J\in\mathcal D^\kappa\) with \(|J|=2^{k+4}\), if we define
\[
    f_J^m
    :=
    f^m\mathbf 1_{J/3}
    \mathbf 1_{\bigcup_{Q\in\mathcal I_k^{\mathrm{sm}}}Q},
\]
then the pointwise central-third partition gives the exact identity
\begin{equation}\label{eq: partition identity}
    f^m
    \mathbf 1_{\bigcup_{Q\in\mathcal I_k^{\mathrm{sm}}}Q}
    =
    \sum_{\kappa\in\{0,1/3,2/3\}}
    \sum_{\substack{J\in\mathcal D^\kappa\\|J|=2^{k+4}}}
    f_J^m.
\end{equation}
Continuing with the theme of spatial localization, if we define
\[
    T_Jh
    :=
    \nu_k*\big(h\cdot \mathbf 1_{J/3}\big),
    \qquad |J|=2^{k+4},
\]
then 
\[
    J/3+I_k\subseteq J
\]
and thus \(T_Jf_J^m\) is supported in \(J\).

Note that, by construction, each dyadic piece satisfies the ``occupancy" estimate
\[
    |\operatorname{supp}f_J^m|
    \lesssim 2^{k-m},
    \qquad |J|=2^{k+4}.
    \]
Indeed, 
\begin{align}
    2^m|\operatorname{supp}f_J^m|
    &\leq
    \sum_{\substack{Q\in\mathcal I_k^{\mathrm{sm}}\\Q\cap J/3\neq\varnothing}}
    \sum_{x\in Q}f(x) \\
    &\lesssim
    \sum_{\substack{Q\in\mathcal I_k^{\mathrm{sm}}\\Q\cap J/3\neq\varnothing}}
    |Q| \\
    &\lesssim 2^k,
\end{align}
since the intervals $Q$ in the sum  are disjoint, meet $J/3$ and have length $\lesssim 2^k.$
For future reference, we call explicit attention to the following estimate: if
\[
    m=\ell_k-s,
    \qquad
    \ell_k=\Big\lfloor\frac{k}{2}\Big\rfloor,
    \qquad
    s \geq 0
\]
then
\begin{equation}\label{eq: occupancy estimate}
    |\operatorname{supp}f_J^{\ell_k-s}|
    \lesssim
    2^{k-\ell_k+s}
    \lesssim
    2^sW_k,
\end{equation}
since \(W_k\approx2^{k/2}\).

\subsection{Short and long altitude gaps}

Notice that the low dyadic contribution can be handled by the standard $\ell^2(\mathbb Z)$ almost-sure maximal estimate (see \eqref{eq: l^2 estimate}), so in what follows we may restrict to indices with $\ell_k -s \geq 0.$ Given an integer \(s\geq 1\), define the set of scales
\[
    \mathcal K_s
    :=
    \Big\{
        k\geq 1:
        0\leq \ell_k-s,
        \quad
        s>100 \log(k+2)
    \Big\},
\]
and the corresponding square functions
\[
    \mathcal S_sf(x)
    :=
    \Big(
        \sum_{k\in\mathcal K_s}
        |\nu_k*f^{\ell_k-s}(x)|^2
    \Big)^{1/2}.
\]
The following lemma separates the short and long gap contributions.
\begin{lemma}\label{l:short-long-localization}
Away from the exceptional set \(X_{\text{high}}\cup E_{\text{CZ}}\) we may bound
\begin{align}\label{e:SL}
    &\sup_{k\geq 1}|\mu_k*f| \\
    &\lesssim
    M_{\text{HL}}f +
    \sum_{\kappa\in\{0,1/3,2/3\}} \sup_{k\geq 1}
    \Big|
        \sum_{0\leq s\leq 100\log(k+2)}
        \sum_{\substack{J\in\mathcal D^\kappa\\|J|=2^{k+4}}}
        T_{J}f^{\ell_k-s}_{J}
    \Big|  +
    \sum_{s\geq 1}\mathcal S_sf.
\end{align}
Moreover, almost surely
\begin{align}\label{e:S2}
    \|\mathcal S_sf\|_{\ell^2}^2
    \lesssim_\omega
    2^{-s/2}\|f\|_{\ell^1}
\end{align}
for all $s \geq 1.$
\end{lemma}
\begin{proof}
Outside the high-value exceptional set \(X_{\text{high}}\), only levels
\(m \le \ell_k\) can contribute. We thus write
\(m=\ell_k-s\), under the conditions \(s\geq 0\) and $\ell_k -s \geq 0$. Off \(E_{\text{CZ}}\), the localization
above decomposes each such level into the functions
\(f^{\ell_k-s}_{J}\). By the partition identity in \eqref{eq: partition identity} and the triangle inequality, 
\begin{align}\label{eq: scale splitting}
    &\sup_{k \geq 1} |\mu_k \ast f| \\
    & \leq \sup_{k \geq 1} |\mathbb{E}(\mu_k)\ast f| + \sum_{\kappa \in\{0,1/3,2/3\}} \sup_{k \geq 1 } \Big| \sum_{0 \leq s \leq 100 \log(k+2)} 
    \sum_{\substack{J\in\mathcal D^\kappa \\|J|=2^{k+4}}}
    T_Jf_J^{\ell_k-s} \Big| \notag\\
    &+ \sup_{k \geq 1 } \sum_{\substack{0 \leq \ell_k -s \\ s > 100\log(k+2)}}
     |\nu_k \ast f^{\ell_k-s} |. \notag
\end{align}
For the second sum, we make use of the embedding $\ell^2 \hookrightarrow \ell^\infty$ and obtain 
\begin{align}\label{eq: square functions through embedding}
    \sup_{k \geq 1 } \sum_{\substack{0 \leq \ell_k -s \\ s > 100\log(k+2)}}
     |\nu_k \ast f^{\ell_k-s} | \leq \sum_{s \geq 1} \sup_{k \in \mathcal{K}_s} |\nu_k \ast f^{\ell_k-s}| \leq \sum_{s \geq 1} \mathcal{S}_sf.
\end{align}
Since, by Lemma \ref{lem: deterministic bound}, the deterministic expectation in \eqref{eq: scale splitting} satisfies
\[
    \sup_{k \geq 1}|\mathbb E(\mu_k)*f|
    \lesssim M_{\text{HL}}f,
\]
it remains only to prove \eqref{e:S2}. Applying Proposition \ref{thm: fourier bound nu_k}, we find that, almost surely and for all $s \geq 1$, 
\begin{align}\label{e:Long1}
    \|\mathcal S_sf\|_{\ell^2}^2
    &=
    \sum_{k\in\mathcal K_s}
    \|\nu_k*f^{\ell_k-s}\|_{\ell^2}^2\\
    &\leq
    \sum_{k\in\mathcal K_s}
    \|\widehat{\nu_k}\|_{L^\infty(\mathbb{T})}^2
    \|f^{\ell_k-s}\|_{\ell^2}^2\\
    &\lesssim_\omega
    \sum_{k\in\mathcal K_s}
    \frac{k}{W_k}
    2^{\ell_k-s}
    \|f^{\ell_k-s}\|_{\ell^1}\\
    &\lesssim 
    2^{-s}
    \sum_{k\in\mathcal K_s}
    k\|f^{\ell_k-s}\|_{\ell^1}.
\end{align}
For fixed \(s\), the map \(k\mapsto \ell_k-s\) has uniformly bounded
multiplicity, and if \(k\in\mathcal K_s\) then
\[
    k+2\leq \exp\Big(\frac{s}{100}\Big) \qquad \Rightarrow \qquad 
    k\lesssim 2^{s/2};
\]
substituting this in \eqref{e:Long1} yields the result.
\end{proof}

The takeaway from this section is that the main difficulty in proving
Theorem \ref{t:main} is the short-gap contribution appearing in
Lemma \ref{l:short-long-localization}. In the following section, we will reduce its control to a
restricted \(\ell^2\) estimate for convolutions with indicators of sets of
bounded size.

\section{The Key Restricted Estimate}
The key quantitative estimate in our paper is the following proposition.

\begin{proposition}\label{p:key}
Fix a constant \(C_1 \geq 1\). Almost surely, there exists a finite random
constant
\[
    C_{\omega,C_1} > 0
\]
such that the following holds.

For every \(k\geq 1\), every shifted-dyadic interval \(J \in \mathcal{D}^\kappa\) with
\[
    |J|=2^{k+4}, \qquad \kappa \in \{0,1/3,2/3\}
\]
and every set \(A\subset J/3\) satisfying
\begin{equation}\label{eq: sparse sets}
    |A|\leq C_1W_k (k+2)^{100},
\end{equation}
one has
\[
    \|\nu_k*\mathbf 1_A\|_{\ell^2}^2
    \leq
    C_{\omega,C_1}
    W_k^{-1}\,|A|.
\]
The estimate is simultaneous over all admissible choices of
\(k, J,\kappa\), and \(A\).
\end{proposition}

\begin{remark}
By Proposition~\ref{thm: fourier bound nu_k} and Plancherel's identity, we have
\begin{equation}\label{eq: fourier bound conv}
    \|\nu_k * \mathbf{1}_A\|_{\ell^2}^2
    \lesssim_\omega
    k\,W_k^{-1}|A|
\end{equation}
almost surely, uniformly over all \(k\geq1\) and all finite sets \(A\). The content of Proposition~\ref{p:key} is that, if one restricts to sets which are sufficiently sparse with respect to the ambient scale, the factor $k$ in the right-hand side of the above inequality may be removed.

%An earlier argument of the second author improves the factor \(k\) to
%\(\log(k+2)\) for sets \(A\) in the cardinality range relevant to the
%preceding proposition; we give a high-level
%sketch of this proof in Appendix~\ref{app:old-argument}. 
\end{remark}
\subsection{Proof of Theorem \ref{t:main} conditional on Proposition \ref{p:key}}
\begin{proof}[Proof of Theorem \ref{t:main} assuming Proposition \ref{p:key}]
By homogeneity, we may set \(\lambda=1\). As before, the low dyadic part is handled by the $\ell^2(\mathbb Z)$ maximal
estimate in \eqref{eq: l^2 estimate}, so throughout the argument below we restrict to indices with $\ell_k - s \ge 0$.

The sets \(X_{\text{high}}\) and \(E_{\text{CZ}}\), together with the
Hardy--Littlewood term arising in Lemma \ref{l:short-long-localization}, have total size
\[
    \lesssim \|f\|_{\ell^1}.
\]
For the long gaps, we bound
\begin{align}
    \Big|\Big\{x : \sum_{s\geq1}\mathcal S_sf(x) \geq1\Big\}\Big|
    &\leq
    \sum_{s\geq1}
    \Big|\Big\{x : \mathcal S_sf(x) \gtrsim s^{-2}\Big\}\Big| \\
    &\lesssim
    \sum_{s\geq1}s^4\|\mathcal S_sf\|_{\ell^2}^2
    \lesssim_\omega \|f\|_{\ell^1}
\end{align}
almost surely. 
It therefore remains to estimate the short-gap expression in
Lemma~\ref{l:short-long-localization}.

For \(s\geq0\) and $\kappa \in \{0, 1/3, 2/3\}$, define the square functions
\[
    S_{s,\kappa}f
    :=
    \Big(
        \sum_{\substack{k\geq 1:\ s\leq100\log(k+2)\\
                        0\leq\ell_k-s\\}}
        \sum_{\substack{J\in\mathcal D^\kappa\\|J|=2^{k+4}}}
        \big|T_Jf_J^{\ell_k-s}\big|^2
    \Big)^{1/2}.
\]
Since $\ell^2(k) \hookrightarrow \ell^\infty(k)$, proceeding as in \eqref{eq: square functions through embedding} implies that the short
part of the \(\kappa\)-grid expression in Lemma \ref{l:short-long-localization} is bounded pointwise by
\[
    \sum_{s\geq0}S_{s,\kappa}f.
\]

For each interval $J$ and each pair $(k, s)$, write
\[
    f_J^{\ell_k-s}=2^{\ell_k-s}\mathbf1_{A_{k,s,J}};
\]
the occupancy estimate in \eqref{eq: occupancy estimate} yields the bound
\[
    |A_{k,s,J}|
    \lesssim W_k2^s \leq W_k (k+2)^{100}
\]
by the short-range condition. Consequently, Proposition~\ref{p:key} applies, after enlarging the constant $C_1$ in the statement once and for all. It yields the almost-sure estimate
\begin{align}
    \|T_Jf_J^{\ell_k-s}\|_{\ell^2}^2
    &=
    2^{2(\ell_k-s)}
    \|\nu_k*\mathbf1_{A_{k,s,J}}\|_{\ell^2}^2 \\
    &\lesssim_\omega
    2^{2(\ell_k-s)}W_k^{-1}
    |A_{k,s,J}| \\
    &\lesssim
    2^{-s}
    \|f_J^{\ell_k-s}\|_{\ell^1}.
\end{align}

Summing over \(k\) and the intervals \(J\), we obtain
\begin{align}\label{e:moderately-square}
    \|S_{s,\kappa}f\|_{\ell^2}^2
    \lesssim_\omega
    2^{-s}
    \sum_{\substack{k\geq 1:\ 0\leq s\leq100\log(k+2)\\
                    0\leq\ell_k-s\\
                    }}
    \|f^{\ell_k-s}\|_{\ell^1} \lesssim 2^{-s} \|f \|_{\ell^1};
\end{align}
arguing as above completes the proof, taking into account previous reductions.
\end{proof} 

\section{Reduction to a Fixed-Support Estimate}
The aim of this section is to reduce Proposition~\ref{p:key} to a convolution estimate for a random signed measure with fixed support on a finite cyclic group, stated below as Proposition \ref{prop: fixed support prop}.

Fixing \(k\geq 1\), a
localization interval \(J\) with \(|J|=2^{k+4}\) and an admissible set
\(A\subset J/3\),
choose
\[
    G:=\mathbb Z/N\mathbb Z,
    \qquad
    N:=2^{k+8}.
\]
After translating \(J\), the supports of \(\mathbf 1_A\), \(\nu_k\) and
\(\nu_k*\mathbf 1_A\) are contained in a single fundamental domain of \(G\).
In particular, the cyclic convolution agrees with the original convolution on
\(\mathbb Z\), and no ``wrap-around" occurs. Since \(N\approx 2^k\), the choice of $N$ affects only absolute constants. 

So, henceforth we work on
\[
    G=\mathbb Z/N\mathbb Z
\]
with 
\[
    W:=W_k\approx N^{1/2}.
\]
Identifying the dyadic interval \(I_k\) with its image in \(G=\mathbb Z/N\mathbb Z\), we define the random function
$$
    Y:G\to[-1,1],
    \qquad
    Y(n):=
    \begin{cases}
        Y_n, & n\in I_k,\\
        0, & n\notin I_k;
    \end{cases}
$$
in this setting, we will also work under the convention that the selector probabilities satisfy $$\sigma_n = \begin{cases}
    n^{-1/2}, \quad n \in I_k \\
    0, \quad n \notin I_k.
\end{cases}$$

For functions $h : G \to \mathbb{C}$, we use
\begin{align}
    \| h \|_p := \Big(\sum_{x \in G} |h(x)|^p \Big)^{1/p}
\end{align}
for all $p \geq 1$, and $$\|h\|_\infty := \max_{x \in G} |h(x)|.$$
We write \(\ell^2(G)\) for the Hilbert space of functions \(h: G\to\mathbb C\), equipped with the norm \(\|\cdot\|_2\) and the standard inner product
\[
\langle u,v\rangle_{\ell^2(G)}
:=
\sum_{x\in G}u(x)\overline{v(x)}.
\]

We now symmetrize the random function $Y$ in order to separate the randomness of its support from that of its values. We let $Y'$ be an independent copy of $Y$; a central object of study will be the random set
\[
    S := \{ n \in G : Y(n) \neq Y'(n) \}.
\]
In particular, by construction, we may write
\[
   Z:=  Y - Y' = \sum_{u\in S}\varepsilon_u\delta_u,
\]
where $(\varepsilon_u)_{u \in S}$ is a vector of independent Rademacher signs supported on $S$.

We record the relevant probabilistic properties of $S$ in the following lemma.

\begin{lemma}\label{lem: S properties}
There exists an absolute constant $C_0 > 0$ such that, for every $k \ge 1$,
\[
    \mathbb{P}
    \Big(
        |S| \le C_0 W,
        \;
        \max_{t \neq 0} r_S(t) \le C_0(k+2)^2
    \Big)
    \ge 1 - N^{-200},
\]
where
\[
    r_S(t) := \bigl|\{(u,v) \in S^2 : u - v = t\}\bigr| = \mathbf{1}_S * \mathbf{1}_{-S}(t).
\]
\end{lemma}
\begin{proof}
For $n \in G$, let
\[
    \eta_n := \mathbf{1}_{\{Y(n) \neq Y'(n)\}}, 
    \]
so that
\[
    |S| = \sum_{n \in G} \eta_n.
\]
The variables $(\eta_n)_{n \in G}$ are independent Bernoulli variables, with
\[
    p_n := \mathbb{E}\, \eta_n = 2\sigma_n(1-\sigma_n).
\]
Since $\sigma_n \approx 2^{-k/2} \approx W^{-1}$,
\[
     \mathbb E |S| = \sum_{n \in G} p_n \lesssim W.
\]
Similarly, we have the variance estimate $$\operatorname{Var}(|S|) = \sum_{n \in G} \operatorname{Var}(\eta_n) = \sum_{n \in G}p_n(1-p_n) \lesssim W,$$ so the first point follows by Bernstein's inequality. %for all $\lambda > 0$ 
%\begin{align}
%   \mathbb{P}\bigl( \big||S|-\mathbb E |S|\big| > \lambda) &\lesssim \operatorname{exp}\Big( -c\operatorname{min}\Big( \frac{\lambda^2}{W}, \lambda \Big) \Big)
%\end{align}
%for some absolute constant $c > 0.$ Choosing $C_0$ sufficiently large and %taking $\lambda \approx C_0W$, we obtain 
%\begin{equation}\label{eq: bound on size of S}
%\mathbb P (|S| > C_0 W) \lesssim e^{-c'C_0 W} \leq N^{-201}
%\end{equation}
%for all $k$, after possible enlarging $C_0$ to deal with small scales. 

We now bound the difference multiplicities. Fix $t \in G$, $t \neq 0$. By definition, we have that 
\[
    r_S(t) = \sum_{\substack{n \in I_k \\ n-t \in I_k}} \eta_n \eta_{n-t}.
\]
We may partition $$I_k = I_k^1 \cup I_k^2$$ so that $$I_k^1 \cap (I_k^1-t) = \emptyset, \qquad I_k^2 \cap (I_k^2 -t) = \emptyset.$$ 
Indeed, since \(N=2^{k+8}\), every nonzero \(t\in G\) has even order. We may therefore decompose \(G\) into the cyclic arithmetic progressions generated by \(t\), assign the points in each progression alternately to two classes, and then restrict this partition to \(I_k\).

We thus have that $$r_S(t) = \sum_{\substack{n \in I_k^1 \\ n-t \in I_k}} \eta_n \eta_{n-t} + \sum_{\substack{n \in I_k^2 \\ n-t \in I_k}} \eta_n \eta_{n-t}.$$ For each \(j\in\{1,2\}\), by construction the pairs
\[
\{n,n-t\},
\qquad
n\in I_k^j,\quad n-t\in I_k,
\]
are pairwise disjoint.
Since the variables \((\eta_n)_{n\in G}\) are independent, it follows that the products
$
\eta_n\eta_{n-t}
$
appearing within each of the two sums are mutually independent.

Moreover,
\[
    \mathbb{E}(\eta_n \eta_{n-t}) = p_n p_{n-t} \lesssim W^{-2},
\]
and there are at most $|I_k| \lesssim W^2$ summands. If we let $$R_j(t) := \sum_{\substack{n \in I_k^j \\ n-t \in I_k}}\eta_n \eta_{n-t}, \qquad j \in \{1, 2\},$$ then 
\[
    \mathbb{E} R_j(t) \lesssim 1
\]
and $$\operatorname{Var}(R_j(t)) \lesssim 1$$
for all $j \in \{1, 2\}.$
If $C_0$ is sufficiently large, applying Bernstein's inequality to each $R_j(t)$ yields
\[
    \mathbb{P}\bigl( r_S(t) > C_0(k+2)^2 \bigr) \lesssim \exp\bigl( -cC_0(k+2)^2 \bigr)
\]
for all sufficiently large $k$.

Taking a union bound over the possible nonzero values of $t \in G$ yields
\[
    \mathbb{P}\Big( \max_{t \neq 0} r_S(t) > C_0(k+2)^2 \Big)
    \lesssim N \exp\bigl( -cC_0(k+2)^2 \bigr)
\]
for all $k$, after possibly increasing $C_0$ to deal with the small scales; the result follows by taking a union bound.
\end{proof}

We now state the fixed-support Proposition which will allow us to establish Proposition \ref{p:key}.
\begin{proposition}\label{prop: fixed support prop}
    Let $C_0 > 0$ and $C_1 \geq 1$ be arbitrary finite constants. Let
\[
G=\mathbb Z/N\mathbb Z,
\qquad
N =  2^{k+8},
\qquad
W\approx N^{1/2}.
\]
Let \(S\subseteq G\) be a fixed set satisfying
\[
|S|\le C_0 W
\]
and
\begin{equation}\label{eq: difference bound}
D = D_S := \max_{t\neq 0} r_S(t)
\le C_0 (k+2)^2,
\qquad
r_S(t)
:=
\bigl|\{(u,v)\in S^2:u-v=t\}\bigr|.
\end{equation}
Let \((\varepsilon_u)_{u\in S}\) be a vector of independent Rademacher
variables supported on $S$ and set
\[
Z=\sum_{u\in S}\varepsilon_u\delta_u.
\]

Then there exist finite constants \(C\) and \(k_0\), depending only
on \(C_0\) and $C_1$, such that for every \(k\ge k_0\),
\[
\mathbb P_\varepsilon\Big(
\sup_{\substack{A\subseteq G\\
1\le |A|\le C_1W(k+2)^{100}}}
\frac{\|Z*1_A\|_{\ell^2(G)}^2}{W|A|}
\le C
\Big)
\ge 1-N^{-100}.
\]
\end{proposition}
\subsection{Proof of Proposition \ref{p:key} conditional on Proposition \ref{prop: fixed support prop}}
\begin{proof}[Proof of Proposition \ref{p:key} assuming Proposition \ref{prop: fixed support prop}] 
Fix $k \ge 1$, a localization interval $J$ with $|J| = 2^{k+4}$, and identify the relevant functions with functions on
\[
    G = \mathbb{Z}/N\mathbb{Z}, \qquad N = 2^{k+8},
\]
as above.
For $y \in \ell^2(G)$, define
\[
    F(y) := \sup_{\substack{A \subseteq G \\ 1 \le |A| \le C_1W(k+2)^{100}}} \frac{\|y \ast \mathbf{1}_A\|_2^2}{W|A|}.
\]
As above, let $Y'$ be an independent copy of $Y$ and let
\[
    S = \{ n \in G : Y(n) \neq Y'(n) \}.
\]
Then
\[
   Y - Y' = \sum_{u \in S} \varepsilon_u \delta_u
\]
where, conditional on $S$, the variables $(\varepsilon_u)_{u \in S}$ are independent Rademacher signs.

By Lemma~\ref{lem: S properties}, with probability at least $1 - N^{-100}$,
\[
    |S| \le C_0 W, \qquad \max_{t \neq 0} r_S(t) \le C_0(k+2)^2.
\]
Conditional on any support $S$ satisfying these bounds, Proposition~\ref{prop: fixed support prop} yields
\[
    \mathbb{P}_\varepsilon \big( F(Y-Y') \le C \mid S \big) \ge 1 - N^{-100};
\]
combining these two results yields
\[
    \mathbb{P}_{Y,Y'} \big(F(Y-Y') > C \big)\lesssim N^{-100}.
\]

We now note that the functional $F$ is convex, since it is a supremum of convex functions. Moreover,
\[
    \mathbb{E}_{Y'}[Y - Y' \mid Y] = Y,
\]
because $Y'$ is independent of $Y$ and has mean zero. Therefore, the conditional form of Jensen's inequality yields the pointwise bound 
\[
    F(Y) \le \mathbb{E}_{Y'} \left[ F(Y-Y') \mid Y \right].
\]

Using \(|Y-Y'|\le1\) and Young's inequality, we have the crude  deterministic bound
$$
F(Y-Y')\le C_1N(k+2)^{100}.
$$
Hence, splitting the conditional expectation according to whether \(F(Y-Y')\le C\) yields
\begin{align}
F(Y) &\le \mathbb E_{Y'}\bigl[F(Y-Y')\mid Y\bigr] \\
&\le C+ C_1N(k+2)^{100} \mathbb P_{Y'} \big(F(Y-Y')>C \,\big|\, Y \big).
\end{align}
It follows that
$$
\mathbb P_Y\big(F(Y)>C+1\big)
\leq 
\mathbb P_Y \Big(
\mathbb P_{Y'}
\big(
F(Y-Y')>C
\,\big|\,
Y
\big)
\geq
\frac{1}{C_1N(k+2)^{100}}
\Big).
$$

Therefore, by Markov's inequality,
\begin{align}
\mathbb P_Y\big(F(Y)>C+1\big)
&\le
C_1N(k+2)^{100}
\mathbb P_{Y,Y'}\big(F(Y-Y')>C\big) \\
&\lesssim_{C_1}
(k+2)^{100}N^{-99} \\
& \lesssim_{C_1} N^{-50}.
\end{align}

The Borel--Cantelli lemma therefore implies that, almost surely, there exists a finite random threshold $k_* = k_*(\omega)$ such that
\[
    F(Y) \lesssim_{\omega, C_1} 1
\]
for every $k \ge k_*$.

Returning from the cyclic model to $\mathbb{Z}$, we conclude that almost surely
\[
    \|\nu_k \ast \mathbf{1}_A\|_{\ell^2(\mathbb{Z})}^2 = W^{-2} \|Y \ast \mathbf{1}_A\|_2^2 \lesssim_{\omega, C_1} W^{-1}|A|
\]
for every $k \ge k_*$ and every admissible $A \subseteq J/3$. Enlarging the implicit random constant to account for the finitely many values $k < k_*$ proves Proposition~\ref{p:key}.
\end{proof}
With this in mind, the remainder of the paper will be devoted to proving Proposition \ref{prop: fixed support prop}.
\section{Overview of the Proof of Proposition \ref{prop: fixed support prop}}\label{sec: proof overview}
At its core, the proof of Proposition \ref{prop: fixed support prop} is organized around a sampling argument whose purpose is to replace a very large family of quadratic forms by a much smaller family of well-controlled empirical representatives.

The first step involves expanding the convolution square and performing a series of preliminary reductions. Specifically, note that, fixing $A \subseteq G$, we may write 
\begin{equation}\label{eq: conv expansion}
\frac{\|Z \ast \mathbf{1}_A\|_2^2}{W|A|} = \frac{|S|}{W} + \frac{1}{W |A|}\sum_{\substack{u, v \in S \\ u \neq v}} \varepsilon_u \varepsilon_v r_A(u-v),
\end{equation}
where $$r_A(t) = |\{(b, c ) \in A^2: b-c = t\}|.$$
If $|A|$ is small, specifically $$|A| \le \frac{W}{(k+2)^2},$$ then the quadratic form in \eqref{eq: conv expansion} may be controlled directly, making use of the difference-count bound $$D = \max_{t \neq 0} r_S(t) \lesssim (k+2)^2$$ in the statement of Proposition \ref{prop: fixed support prop}. 

Moreover, we may express the off-diagonal term in \eqref{eq: conv expansion} in the form 
\begin{equation}\label{eq: collision expansion}
\sum_{\substack{u, v \in S \\ u \neq v}} \varepsilon_u \varepsilon_v r_A(u-v) = \sum_{(u, v, b, c ) \, \operatorname{a \, collision}} \varepsilon_u \varepsilon_v,
\end{equation}
where a \emph{collision} is defined as a tuple $(u, v, b, c) \in S \times S \times A \times A$ satisfying $u-v = b-c$ and $u \neq v$. Associated with a collision $(u, v, b, c)$, we may define two \emph{centers} $$y := b+v = c+u, \qquad z := b-u = c-v,$$ so that $$b, c \in A \cap (y-S) \quad \operatorname{and} \quad b, c \in A \cap (z+S).$$ We will fix a parameter $\tau > 0$ and say that a center $x \in G$ is \emph{heavy} if $$|A \cap (x-S)| >\tau \quad \operatorname{or} \quad |A \cap (x + S)| > \tau.$$ Choosing $\tau$ appropriately allows us to bound the total normalized contribution in \eqref{eq: collision expansion} coming from collisions associated with heavy centers: $$\frac{1}{Wa} \bigl|\sum_{\substack{(u, v, b, c) \, \operatorname{heavy \,collision}}}\varepsilon_u \varepsilon_v\bigr| \lesssim \frac{|S|}{W}.$$

After this ``trimming" procedure, the remaining contribution may be expressed as a quadratic form $$\varepsilon^{\mathsf T} M_A \varepsilon,$$ where $M_A$ is a real symmetric matrix with zeros on its diagonal. Specifically, the matrix $M_A$ is formed as an average $$M_A = \frac{|A|}{W}\,\mathbb E_{b,c\in A} H_{b,c},$$ where the matrix $H_{b, c}$ records the surviving collisions associated with the pair $(b,c) \in A^2.$

The trimming step and the cardinality restrictions on $A$ allow us to establish useful deterministic bounds on the operator and Frobenius norms of the matrices $H_{b, c}$; see Lemma \ref{lem: frobenius bounds} and Lemma \ref{lem: operator norm bounds}.

If \(A\) were fixed, the estimates would allow us to conclude by applying Hanson--Wright (stated as Lemma \ref{lem:HW} below) to control the quadratic form $\varepsilon^{\mathsf T} M_A \varepsilon$ directly. The main difficulty is uniformity in \(A\), since the collection of admissible sets is far too large for a direct union bound over the corresponding matrices \(M_A\). The key idea for overcoming this difficulty is an auxiliary sampling procedure, which will be developed in Section \ref{sec:sampling}. The idea is to replace the full average defining \(M_A\) by an empirical average over a small random sample from \(A\), while preserving the analytic information needed for the argument. Specifically, we fix a parameter $m$ (which will be taken to be polynomial in $k$) and introduce the sampled matrices $$\widehat M_A := \frac{|A|}{Wm(m-1)}
\sum_{i\neq j}H_{b_i,b_j},
$$
where $b_1, \dots, b_m$ are chosen uniformly at random from $A.$ Notice that 
\[
\mathbb E_{\mathrm{aux}}\widehat M_A=M_A,
\]
so that $\widehat M_A$ should be viewed as a random empirical approximation to \(M_A\); here and throughout, we use $\mathbb E_{\mathrm{aux}}, \ \mathbb P_{\mathrm{aux}}$ to denote random behavior with respect to this auxiliary sampling.

We will say that a sampled matrix \(\widehat M_A\) is ``good" when its Frobenius and operator norms satisfy the bounds required in the later Hanson--Wright argument. For each fixed \(A\), the auxiliary sampling produces a good matrix with overwhelming probability; see Lemma \ref{lem:sampling}. Conceptually, this is an approximation statement. Indeed, since \(M_A\) itself already has small norms, a typical empirical matrix \(\widehat M_A\) inherits the same favorable Frobenius and operator-norm profiles with high probability. The adjective good is therefore meant to record precisely the norm information that a typical empirical representative of \(M_A\) possesses and that will be needed in the next stage of the proof.

Although \(M_A\) depends on the entire set \(A\), a sampled matrix \(\widehat M_A\) is determined by much less information. Indeed, once the cardinality $|A|$ and the set of heavy centers associated to $A$ are specified, it is determined by the sampled tuple \(b_1,\ldots,b_m\). Since \(m\) and the size of the set of heavy centers are polynomial in \(k\), the logarithm of the number of possible sampled matrices is itself only polynomial in \(k\). Thus the sampling procedure does not merely approximate \(M_A\), it compresses the enormous family of matrices indexed by \(A\) into a family which is small enough to permit a union bound.

We therefore consider the collection of all good sampled matrices that can arise from an admissible subset \(A\), apply Hanson--Wright to each of them, and take a union bound. This produces a single event in the Rademacher signs on which each element of the set 
\[
\{ |\varepsilon^{\mathsf T}\widehat M_A \varepsilon| : \widehat M_A \text{ good}\}
\]
is bounded uniformly. Using that a sampled matrix is good with high probability together with the identity $$\mathbb E_{\operatorname{aux}} \, \varepsilon^{\mathsf T} \widehat{M}_A \varepsilon = \varepsilon^{\mathsf T} M_A \varepsilon$$ allows us to control the relevant quadratic form uniformly over all admissible sets $A$; see, for instance, Corollary \ref{cor:sampling-transfer}.

Thus the sampling procedure performs three interconnected tasks: it produces empirical approximations of \(M_A\) which, with overwhelming probability, have the required Frobenius and operator-norm bounds; it reduces the complexity of the family from all admissible sets \(A\) to a manageable collection of possible empirical matrices; and its exact unbiased nature allows the resulting uniform Hanson--Wright estimate to be transferred back to the original matrices. This combination of approximation, entropy reduction, and averaging, is the mechanism that converts uniform estimates for each individual \(A\) into a single estimate holding simultaneously for every admissible \(A\). 

Finally, we remark that this use of sampling in seemingly deterministic problems is not novel, see for instance the empirical method of Maurey found in \cite{Pisier1980Maurey}, the proofs of Croot--Sisask almost-periodicity \cite{croot2010probabilistictechniquefindingalmostperiods}, and the probabilistic proofs of the restricted isometry property; see, for instance, Rudelson and Vershynin \cite{RudelsonVershyninSparse}. In all cases, including ours, the underlying philosophy is that a comparatively small random sample can retain the analytic information needed to control a much larger structured family. In the contrapositive: should any extremizing or outlier event occur, it must do so ``often."

\section{A Sampling Principle for Families of Quadratic Forms}\label{sec:sampling}
In this section we record an abstract sampling lemma that will be applied in the proof of Proposition~\ref{prop: fixed support prop}. We first introduce two key probabilistic results that will be used in what follows. 
\subsection{Probabilistic tools}
\begin{lemma}[Hanson--Wright inequality for Rademacher signs]\label{lem:HW}
There exists an absolute constant \(C_{\mathrm{HW}}>0\) such that the following holds. Let \(M\) be a fixed real symmetric matrix with zero diagonal, and let \(\varepsilon\) be a vector of independent Rademacher random variables. Then, for every \(t\ge1\),
$$
\mathbb P\Big(
|\varepsilon^{\mathsf T}M\varepsilon|
>
C_{\mathrm{HW}}
\big(
\sqrt{t}\,\|M\|_F
+
t\,\|M\|_{\mathrm{op}}
\big)
\Big)
\le
2e^{-t}.
$$

\end{lemma}
The preceding estimate is a form of the Hanson-Wright inequality for Rademacher random variables. It follows from \cite[Theorem 1.1]{rudelson2013hansonwrightinequalitysubgaussianconcentration}, since the zero diagonal condition makes the expectation term equal to zero. 

\begin{lemma}[Matrix-valued \(U\)-statistic inequality]
\label{lem:U statistic}
There exists an absolute constant \(C>0\) such that the following
holds. Let \(m,q\in\mathbb N\) and let
\[
    \xi_1,\ldots,\xi_m
\]
be i.i.d.\ random variables taking values in \(\{1,\ldots,q\}\).
Let
\[
    F:\{1,\ldots,q\}\times\{1,\ldots,q\}
      \longrightarrow \operatorname{Sym}(d\times d)
\]
be a function taking values in the space of real symmetric \(d\times d\) matrices, satisfying
\[
    F(x,y)=F(y,x)
\]
for all \(x,y\). Assume moreover that, for every
\(x,y\in\{1,\ldots,q\}\),
\[
    \mathbb E_{\xi_1}F(x,\xi_1)=0,
    \qquad
    \|F(x,y)\|_{\mathrm{op}}\leq B_U,
\]
and
\[
    \mathbb E_{\xi_1}\bigl[F(x,\xi_1)^2\bigr]
    \preceq v^2I
\]
for some \(B_U,v>0\).

Then, for every \(u\geq1\), with probability at least \(1-e^{-u}\),
\[
    \Big\|
        \sum_{i\neq j}F(\xi_i,\xi_j)
    \Big\|_{\mathrm{op}}
    \leq
    C\left[
        mv\bigl(u+\log(ed)\bigr)
        +
        B_U\sqrt m\,
        \bigl(u+\log(ed)\bigr)^{3/2}
    \right].
\]
\end{lemma}
The above estimate is a direct consequence of \cite[Remark~3.11, equation~(3.9)]{minsker2019momentinequalitiesmatrixvaluedustatistics}.

\subsection{Sampling}
Throughout this subsection, we work on the finite cyclic group
\[
    G:=\mathbb Z/N\mathbb Z,
    \qquad
    N\approx W^2\approx 2^k,
\]
and fix a sufficiently large constant $B>0$. In what follows, all lower thresholds for $k$ may depend on $B$.

Let $A \subseteq G$ be a finite nonempty set. Suppose that
\[
    \{H_{b,c}:b,c\in A\}\subseteq\operatorname{Sym}(d\times d),
    \qquad d\lesssim W,
\]
are real symmetric matrices with zero diagonal, and that
\[
    H_{b,c}=H_{c,b}.
\]
Set
\[
    Q_A(b):=\mathbb E_cH_{b,c},
    \qquad
    \overline H_A:=\mathbb E_bQ_A(b),
    \qquad
    M_A:=\lambda_A\overline H_A,
\]
where $b,c$ are independent and uniformly distributed on $A$. Assume
that
\[
    0<\lambda_A\lesssim (k+2)^{100}
\]
and that
\begin{align}
    &\mathbb E_b\|Q_A(b)\|_F^2 \lesssim \frac{(k+2)^{100}}{W},
        \label{eq:sampling-first-hyp}\\
    &\mathbb E_{b,c}\|H_{b,c}\|_F^2 \lesssim (k+2)^{100},
        \label{eq:sampling-second-hyp}\\
    &\sup_{b,c}\|H_{b,c}\|_{\mathrm{op}} \leq1,
        \label{eq:sampling-op-hyp}\\
    &\mathbb E_cH_{b,c}^2 \preceq (k+2)^{-6B}I
        \qquad\text{for every }b\in A.
        \label{eq:sampling-variance-hyp}
\end{align}
We remark that the fixed exponents $100$ in the hypotheses above and
$102$ in Corollary~\ref{cor:sampling-transfer} are chosen to match the
application in the next section. They arise from the cardinality and
difference-count bounds in the proof of
Proposition~\ref{prop: fixed support prop} and are not intrinsic to the
sampling argument. More generally, the same argument accommodates any
fixed polynomial losses, after choosing $B$ sufficiently large in terms
of the corresponding exponents.

We now set
\begin{equation}\label{eq: m def}
    m:=\left\lceil(k+2)^{15B}\right\rceil
\end{equation}
and draw $b_1,\ldots,b_m$ independently and uniformly from $A$. We define the sampled matrix
\begin{equation}\label{eq:sampling-matrix}
    \widehat M_A
    :=
    \frac{\lambda_A}{m(m-1)}
    \sum_{i\neq j}H_{b_i,b_j},
\end{equation}
which by construction satisfies
\[
    \mathbb E_{\mathrm{aux}}\widehat M_A=M_A.
\]
Indeed, the indices $i,j$ in each summand are distinct, although the
sampled points $b_i,b_j$ themselves may coincide.

We call a sampled matrix \emph{good} if
\[
    \|\widehat M_A\|_F<(k+2)^{-9B},
    \qquad
    \|\widehat M_A\|_{\mathrm{op}}<(k+2)^{-17B}.
\]

\begin{lemma}\label{lem:sampling}
If $B$ is sufficiently large, then, uniformly
in $A$ and for all sufficiently large $k$,
\begin{equation}\label{eq:sampling-norm-bounds}
    \mathbb P_{\mathrm{aux}}\left(
        \|\widehat M_A\|_F\geq(k+2)^{-9B}
        \ \text{or}\
        \|\widehat M_A\|_{\mathrm{op}}\geq(k+2)^{-17B}
    \right)
    \lesssim (k+2)^{-10B}.
\end{equation}
\end{lemma}

\begin{proof}
By Jensen's inequality, we bound the deterministic norm 
\[
    \|\overline H_A\|_F^2
    \leq
    \mathbb E_b\|Q_A(b)\|_F^2
    \leq
    \frac{(k+2)^{100}}{W}.
\]
For all $b,c \in A$, we now decompose 
\begin{equation}\label{e:center}
    H_{b,c}
    =
    H^0_{b,c}
    +
    g_A(b)
    +
    g_A(c)
    +
    \overline H_A,
\end{equation}
where
\[
    g_A(b):=Q_A(b)-\overline H_A
\]
and
\[
    H^0_{b,c}
    :=
    H_{b,c}-Q_A(b)-Q_A(c)+\overline H_A.
\]
By construction
\[
    \mathbb E_bg_A(b)=0,
\]
and, using the symmetry $H_{b,c}=H_{c,b}$,
\[
    \mathbb E_cH^0_{b,c}=0,
    \qquad
    \mathbb E_bH^0_{b,c}=0.
\]
Consequently,
\begin{equation}\label{eq:sampling-decomposition}
    \widehat M_A
    =
    M_A
    +
    \frac{2\lambda_A}{m}\sum_{i=1}^m g_A(b_i)
    +
    \frac{\lambda_A}{m(m-1)}
    \sum_{i\neq j}H^0_{b_i,b_j}.
\end{equation}

Note that, by construction, the following orthogonality property holds:
\[
    \mathbb E_{b,c}\|H_{b,c}\|_F^2
    =
    \|\overline H_A\|_F^2
    +2\,\mathbb E_b\|g_A(b)\|_F^2
    +\mathbb E_{b,c}\|H^0_{b,c}\|_F^2.
\]
Moreover,
\[
    \mathbb E_b\|g_A(b)\|_F^2
    \leq
    \mathbb E_b\|Q_A(b)\|_F^2
\]
and
\[
    \mathbb E_{b,c}\|H^0_{b,c}\|_F^2
    \leq
    \mathbb E_{b,c}\|H_{b,c}\|_F^2.
\]
Hence
\begin{equation}\label{e:gstat}
    \mathbb E_{\mathrm{aux}}
    \Big\|
        \frac{2\lambda_A}{m}\sum_{i=1}^m g_A(b_i)
    \Big\|_F^2
    =
    \frac{4\lambda_A^2}{m}
    \mathbb E_b\|g_A(b)\|_F^2
    \lesssim
    \frac{(k+2)^{300}}{mW}.
\end{equation}

For the fully degenerate quadratic term, all cross terms vanish unless the two ordered
pairs of indices agree up to reversal. Recalling the decomposition in \eqref{eq:sampling-decomposition}, 
\begin{align*}
    \mathbb E_{\mathrm{aux}}\|\widehat M_A\|_F^2
    &=
    \|M_A\|_F^2
    +
    \frac{4\lambda_A^2}{m}
        \mathbb E_b\|g_A(b)\|_F^2 \\
    &\qquad
    +
    \frac{2\lambda_A^2}{m(m-1)}
        \mathbb E_{b,c}\|H^0_{b,c}\|_F^2\\
    &\lesssim
    (k+2)^{300}
    \left(
        \frac1W+\frac1{mW}+\frac1{m^2}
    \right).
\end{align*}
Since
\[
    m\approx(k+2)^{15B}
\]
and $W$ dominates every fixed power of $k+2$, if $B$ is sufficiently large 
\[
    \mathbb E_{\mathrm{aux}}\|\widehat M_A\|_F^2
    \lesssim
    (k+2)^{-29B}
\]
for all sufficiently large $k$. Markov's inequality therefore yields
\[
    \mathbb P_{\mathrm{aux}}
    \big(
        \|\widehat M_A\|_F\geq(k+2)^{-9B}
    \big)
    \lesssim
    (k+2)^{-11B}.
\]

We next estimate the operator norm. For the deterministic mean,
\[
    \|M_A\|_{\mathrm{op}}
    \leq
    \|M_A\|_F
    \leq
    \frac{(k+2)^{300/2}}{W^{1/2}}
    <
    \frac13(k+2)^{-17B}
\]
for all sufficiently large $k$.

For the first-order term, \eqref{e:gstat} and Markov's inequality yield
\begin{align*}
    &\mathbb P_{\mathrm{aux}}
    \Big(
        \Big\|
            \frac{2\lambda_A}{m}
            \sum_{i=1}^m g_A(b_i)
        \Big\|_{\mathrm{op}}
        \geq
        \frac13(k+2)^{-17B}
    \Big)\\
    &\qquad\lesssim
    \frac{(k+2)^{34B+300}}{mW}
    \lesssim
    (k+2)^{-20B}
\end{align*}
for all sufficiently large $k$.

It remains to control the completely degenerate quadratic term. By operator
Jensen and \eqref{eq:sampling-variance-hyp},
\[
    Q_A(b)^2
    \preceq
    (k+2)^{-6B}I
\]
for all $b \in A$.
Consequently,
\[
    \mathbb E_cQ_A(c)^2
    \preceq
    (k+2)^{-6B}I,
\]
and another application of Jensen yields
\[
    \overline H_A^2
    \preceq
    (k+2)^{-6B}I.
\]
For self-adjoint matrices $X_1,\ldots,X_4$,
\[
    (X_1+\cdots+X_4)^2
    \preceq
    4\sum_{j=1}^4X_j^2.
\]
It follows that
\[
    \|H^0_{b,c}\|_{\mathrm{op}}\leq4,
    \qquad
    \mathbb E_c(H^0_{b,c})^2
    \preceq
    16(k+2)^{-6B}I.
\]

We may therefore apply Lemma~\ref{lem:U statistic} with
\[
    B_{\mathrm U}=4,
    \qquad
    v=4(k+2)^{-3B},
    \qquad
    u=20B\log(k+2).
\]
Since $d\lesssim W\approx2^{k/2}$,
\[
    u+\log(ed)\lesssim_B k+2.
\]
Thus, outside an event of probability at most
\[
    e^{-u}=(k+2)^{-20B},
\]
we have
\begin{align*}
    \Big\|
        \frac{\lambda_A}{m(m-1)}
        \sum_{i\neq j}H^0_{b_i,b_j}
    \Big\|_{\mathrm{op}}
    &\lesssim_B
    (k+2)^{100}
    \Big(
        \frac{(k+2)^{-3B}(k+2)}{m}
        +
        \frac{(k+2)^{3/2}}{m^{3/2}}
    \Big).\\
\end{align*}
Recalling the definition of $m$ in \eqref{eq: m def}, if $B$ is sufficiently large the right-hand side above is bounded by 
\[
    \frac{1}{3} (k+2)^{-17B},
\]
for all sufficiently large $k.$

Combining the Frobenius estimate and the two operator-norm exceptional
estimates yields
\[
    \mathbb P_{\mathrm{aux}}\left(
        \|\widehat M_A\|_F\geq(k+2)^{-9B}
        \ \text{or}\
        \|\widehat M_A\|_{\mathrm{op}}\geq(k+2)^{-17B}
    \right)
    \lesssim
    (k+2)^{-10B},
\]
as claimed.
\end{proof}

\begin{cor}\label{cor:sampling-signs}
Let $\mathcal M$ be a finite family of real symmetric, zero-diagonal
$d\times d$ matrices, fixed independently of the Rademacher vector
$\varepsilon$, and suppose that
\begin{equation}\label{eq: size assumption matrix class}
    \log|\mathcal M|
    \lesssim
    (k+2)^{16B}.
\end{equation}
Call $M\in\mathcal M$ good if
\[
    \|M\|_F<(k+2)^{-9B},
    \qquad
    \|M\|_{\mathrm{op}}<(k+2)^{-17B}.
\]
Then there exists $c_B>0$ such that, for all sufficiently large $k$,
\[
    \mathbb P_\varepsilon
    \left(
        |\varepsilon^{\mathsf T}M\varepsilon|>1
        \text{ for some good }M\in\mathcal M
    \right)
    \leq
    \exp\bigl(-c_B(k+2)^{16B}\bigr).
\]
\end{cor}

\begin{proof}
Choose
\[
    t=C_2(k+2)^{16B},
\]
where $C_2>0$ is sufficiently large compared with the implicit
constant in \eqref{eq: size assumption matrix class}. If $M$ is good, then
\begin{align*}
    C_{\mathrm{HW}}
    \left(
        \sqrt t\,\|M\|_F
        +
        t\|M\|_{\mathrm{op}}
    \right)
    &\lesssim
    (k+2)^{-B}
    \leq1
\end{align*}
for all sufficiently large $k$. By Lemma \ref{lem:HW}, this implies that
\[
    \mathbb P_\varepsilon
    \big(
        |\varepsilon^{\mathsf T}M\varepsilon|>1
    \big)
    \leq
    2e^{-t}.
\]
Taking a union bound over $\mathcal M$ and choosing $C_2$ sufficiently
large proves the claim.
\end{proof}

\begin{cor}\label{cor:sampling-transfer}
Let $\mathcal A$ be a family of sets and, for each $A\in\mathcal A$,
let $\widehat M_A$ be an auxiliary random matrix satisfying
\[
    \mathbb E_{\mathrm{aux}}\widehat M_A=M_A.
\]
Suppose that every possible realization of every $\widehat M_A$
belongs to a common finite family $\mathcal M$ satisfying the
assumptions of Corollary~\ref{cor:sampling-signs}, and that
\[
    \sup_{A\in\mathcal A}
    \mathbb P_{\mathrm{aux}}
    \big(
        \widehat M_A\text{ is not good}
    \big)
    \lesssim
    (k+2)^{-10B}.
\]
Assume moreover that
\[
    \sup_{\varepsilon\in\{-1,1\}^d}
    \sup_{A\in\mathcal A}
    |\varepsilon^{\mathsf T}\widehat M_A\varepsilon|
    \lesssim
    (k+2)^{102}
\]
for every auxiliary sample outcome. If $B$ is sufficiently large, then, outside a set of sign probability at most
\[
    \exp\bigl(-c_B(k+2)^{16B}\bigr),
\]
one has
\[
    \sup_{A\in\mathcal A}
    |\varepsilon^{\mathsf T}M_A\varepsilon|
    \leq2.
\]
\end{cor}

\begin{proof}
Fix a sign vector outside the exceptional event in
Corollary~\ref{cor:sampling-signs}. For each $A\in\mathcal A$, let
\[
    \mathcal G_A
    :=
    \{\widehat M_A\text{ is good}\}.
\]
By the unbiased nature of the auxiliary sampling, splitting the auxiliary expectation over
$\mathcal G_A$ and $\mathcal G_A^c$ and using the two assumptions above yields
\[
    |\varepsilon^{\mathsf T}M_A\varepsilon|
    \leq
    1+C(k+2)^{102-10B}
    \leq2
\]
for all sufficiently large $k$, where $C$ is independent of $A$ and $k$.
Since the sign event is independent of $A$, the estimate holds
simultaneously for every $A\in\mathcal A$.
\end{proof}

We are now prepared to complete the proof of Proposition \ref{prop: fixed support prop}, thereby concluding Theorem \ref{thm: main theorem}.

\section{Proof of Proposition \ref{prop: fixed support prop}}
To prove Proposition~\ref{prop: fixed support prop} we will make
a sequence of reductions and then verify that the resulting family of quadratic forms satisfies the assumptions of the sampling results from
Section~\ref{sec:sampling}.

As in that section, we fix a sufficiently large constant $B>0$.
Recall that
\[
    |S|\leq C_0W,
    \qquad
    D:=\max_{t\neq0}r_S(t)\leq C_0(k+2)^2,
\]
and that
\[
    G=\mathbb Z/N\mathbb Z,
    \qquad
    N\approx W^2\approx2^k.
\]
Thus $W\approx2^{k/2}$ dominates every fixed power of $k+2$ for all
sufficiently large $k$. We remark that all implicit constants below may depend on
$C_0$ and $C_1$.

The conclusion is immediate when $S$ is empty, so we assume
$|S|\geq1$.

\subsection{First reductions and trimming}\label{sec: first reductions}

Fix a nonempty set
\[
    A\subseteq G,
    \qquad
    |A|\leq C_1W(k+2)^{100},
\]
and write
\[
    a:=|A|.
\]
Expanding the convolution and separating the diagonal contribution
yields
\begin{equation}\label{eq: conv expansion in proof}
    \frac{\|Z\ast\mathbf1_A\|_{2}^2}{Wa}
    =
    \frac{|S|}{W}
    +
    \frac1{Wa}
    \sum_{\substack{u,v\in S\\u\neq v}}
        \varepsilon_u\varepsilon_v r_A(u-v),
\end{equation}
where
\[
    r_A(t)
    :=
    \bigl|\{(b,c)\in A^2:b-c=t\}\bigr|.
\]
Since $|S|\lesssim W$, it follows that
\begin{equation}\label{eq: conv explicit}
    \frac{\|Z\ast\mathbf1_A\|_{2}^2}{Wa}
    \lesssim
    1+
    \frac1{Wa}
    \sum_{\substack{u,v\in S\\u\neq v}}
        \varepsilon_u\varepsilon_v r_A(u-v).
\end{equation}

If
\[
    a\leq \frac{W}{(k+2)^2},
\]
then, using $D\lesssim(k+2)^2$, we obtain 
\begin{align}\label{eq: small sets contrib}
    \frac1{Wa}
    \Big|
        \sum_{\substack{u,v\in S\\u\neq v}}
            \varepsilon_u\varepsilon_v r_A(u-v)
    \Big|
    &\leq
    \frac1{Wa}
    \sum_{t\neq0}r_S(t)r_A(t)\\
    &\leq
    \frac{D}{Wa}
    \sum_{t\neq0}r_A(t)\\
    &=
    \frac{D\,a(a-1)}{Wa}
    \lesssim1,
\end{align}
so the desired estimate is immediate in this range; it remains
only to consider sets satisfying
\[
    \frac{W}{(k+2)^2}
    <
    a
    \leq
    C_1W(k+2)^{100}.
\]
In what follows, we refer to sets satisfying the preceding cardinality
restriction as \emph{admissible}.

For such a set $A$, define
\[
    n_-(y):=|A\cap(y-S)|,
    \qquad
    n_+(z):=|A\cap(z+S)|,
    \qquad y,z\in G,
\]
and set
\begin{equation}\label{eq: tau def}
    \tau:=\frac{a}{(k+2)^{6B}}.
\end{equation}
We define the corresponding sets of \emph{heavy centers} by
\[
    J_-:=\{y\in G:n_-(y)>\tau\},
    \qquad
    J_+:=\{z\in G:n_+(z)>\tau\}.
\]

\begin{lemma}[Heavy centers]\label{lem:heavy-centers}
For all sufficiently large $k$ and all admissible sets $A\subseteq G$,
\[
    \sum_{x\in J_\pm}n_\pm(x)^2\lesssim a|S|
\]
and
\[
    |J_\pm|\lesssim(k+2)^{13B}.
\]
The implicit constants depend only on $C_0$ and $C_1$.
\end{lemma}

\begin{proof}
We prove the claim for $J_-$; the proof for $J_+$ is identical.

Let $\mathcal B$ be the matrix whose rows are indexed by $y\in J_-$
and are given by
\[
    \mathcal B_y=\mathbf1_{y-S}.
\]
Then
\[
    n_-(y)
    =
    \langle\mathbf1_{y-S},\mathbf1_A\rangle,
\]
and hence
\[
    \sum_{y\in J_-}n_-(y)^2
    =
    \|\mathcal B\mathbf1_A\|_2^2.
\]
The diagonal entries of $\mathcal B\mathcal B^{\mathsf T}$ are equal
to $|S|$. If $y\neq y'$, then
\[
    \langle\mathbf1_{y-S},\mathbf1_{y'-S}\rangle
    =
    r_S(y-y')
    \leq D.
\]
It follows from the row-sum bound for the operator norm of symmetric matrices that
\[
    \|\mathcal B\mathcal B^{\mathsf T}\|_{\mathrm{op}}
    \leq
    |S|+D|J_-|.
\]
Consequently,
\begin{equation}\label{eq: sum heavy centers}
    E
    :=
    \sum_{y\in J_-}n_-(y)^2
    \leq
    a\bigl(|S|+D|J_-|\bigr).
\end{equation}

Since every $y\in J_-$ satisfies $n_-(y)>\tau$, we also have
\begin{equation}\label{eq: J_- bound}
    |J_-|\tau^2\leq E,
\end{equation}
which combined with \eqref{eq: sum heavy centers} yields
\[
    E
    \leq
    a|S|+\frac{aD}{\tau^2}E.
\]
By the definition of $\tau$ and since $A$ is admissible,
\[
    \frac{aD}{\tau^2}
    \lesssim
    \frac{(k+2)^{12B+4}}{W}
    \leq\frac12
\]
for all sufficiently large $k$, which implies that
\[
    E\lesssim a|S|.
\]
Finally, 
\[
    |J_-|
    \lesssim
    \frac{a|S|}{\tau^2}
    =
    \frac{|S|}{a}(k+2)^{12B}
    \lesssim
    (k+2)^{12B+2}
    \lesssim
    (k+2)^{13B},
\]
provided $B$ is sufficiently large. This proves the claim.
\end{proof}

We now remove the collisions associated with the heavy centers. Recall
that a collision $(u,v,b,c) \in S \times S \times A \times A$ satisfies
\[
    u\neq v,
    \qquad
    u-v=b-c,
\]
and has associated centers
\[
    y:=b+v=c+u,
    \qquad
    z:=b-u=c-v.
\]
For each fixed ``minus-center" $y$, there are at most $n_-(y)^2$
corresponding collisions, and similarly there are at most $n_+(z)^2$
collisions having a fixed ``plus-center" $z$, so
Lemma~\ref{lem:heavy-centers} yields
\begin{equation}\label{eq: heavy collisions contrib}
    \frac1{Wa}
    \Big|
        \sum_{\substack{(u,v,b,c)\ \mathrm{a\ collision}\\
                        y\in J_-\ \mathrm{or}\ z\in J_+}}
            \varepsilon_u\varepsilon_v
    \Big|
    \lesssim
    \frac{|S|}{W}
    \lesssim1.
\end{equation}

It therefore remains only to consider collisions for which neither
center is heavy. For $b,c\in A$, define the matrix $L_{b, c}$, whose entries satisfy 
\[
    L_{b,c}(u,v)
    :=
    \mathbf1_{\{u\neq v\}}
    \mathbf1_{\{u-v=b-c\}}
    \mathbf1_{\{b+v\notin J_-\}}
    \mathbf1_{\{b-u\notin J_+\}},
    \qquad u,v\in S,
\]
and set
\[
    H_{b,c}
    :=
    \frac12
    \bigl(L_{b,c}+L_{b,c}^{\mathsf T}\bigr).
\]
We see that every $H_{b,c}$ is symmetric and has zero diagonal and
\[
    L_{c,b}=L_{b,c}^{\mathsf T},
    \qquad
    H_{c,b}=H_{b,c},
\]
since reversing a collision preserves its two centers.

We now set
\[
    \lambda_A:=\frac aW,
    \qquad
    \overline H_A:=\mathbb E_{b,c}H_{b,c},
\]
and define
\begin{equation}\label{eq: M_A def}
    M_A
    :=
    \lambda_A\overline H_A
    =
    \frac1{Wa}
    \sum_{b,c\in A}H_{b,c}.
\end{equation}
Since
\[
    \varepsilon^{\mathsf T}H_{b,c}\varepsilon
    =
    \varepsilon^{\mathsf T}L_{b,c}\varepsilon,
\]
the quadratic form
\[
    \varepsilon^{\mathsf T}M_A\varepsilon
\]
is precisely the normalized contribution of the surviving collisions.
In particular, recalling \eqref{eq: conv expansion in proof} and \eqref{eq: heavy collisions contrib}, we deduce that 
\begin{equation}\label{eq:trimmed-conv-expansion}
    \frac{\|Z\ast\mathbf1_A\|_2^2}{Wa}
    \lesssim
    1+
    |\varepsilon^{\mathsf T}M_A\varepsilon|,
\end{equation}
uniformly over all admissible sets $A.$ We have thus reduced matters to bounding the quadratic forms above. 

\subsection{Sampling and the universal sign event}

We now apply the sampling principle from Section~\ref{sec:sampling}.
Fix
\[
    m:=\left\lceil(k+2)^{15B}\right\rceil
\]
and let $b_1,\ldots,b_m$ be independent uniformly distributed points
of $A$. We form the sampled matrix
\[
    \widehat M_A
    :=
    \frac{\lambda_A}{m(m-1)}
    \sum_{i\neq j}H_{b_i,b_j};
\]
as alluded to above,
\[
    \mathbb E_{\mathrm{aux}}\widehat M_A=M_A.
\]

The deterministic estimates established below verify the hypotheses of
Lemma~\ref{lem:sampling}; consequently, if
$\mathcal G_A$ denotes the auxiliary event on which $\widehat M_A$ is
good, then
\begin{equation}\label{eq: P_aux norm bounds}
    \mathbb P_{\mathrm{aux}}(\mathcal G_A^c)
    \lesssim
    (k+2)^{-10B}.
\end{equation}

We next construct the finite family of sampled matrices over which the
universal sign estimate will be applied. A sampled matrix
$\widehat M_A$ is completely determined by the data
\[
    a=|A|,
    \qquad
    (b_1,\ldots,b_m),
    \qquad
    J_-,
    \qquad
    J_+,
\]
where by Lemma~\ref{lem:heavy-centers}, there exists a constant $C_2>0$,
depending only on $C_0$, such that
\[
    |J_\pm|\leq C_2(k+2)^{13B}.
\]
Let $\mathcal M_S$ denote the family of all matrices obtained from the
preceding formulas by allowing
\[
    1\leq a\leq N,
    \qquad
    b_1,\ldots,b_m\in G,
\]
and
\[
    J_-,J_+\subseteq G,
    \qquad
    |J_\pm|\leq C_2(k+2)^{13B}.
\]
We allow all such choices, including choices that do not arise from
an actual admissible set. Thus $\mathcal M_S$ is an enlargement of
the collection of sampled matrices that may occur in the proof.

There are at most $N^m$ choices for the ordered sample
$(b_1,\ldots,b_m)$ and at most
\[
    (N+1)^{C_2(k+2)^{13B}}
\]
choices for each of the sets $J_-$ and $J_+$. Therefore, 
\[
    |\mathcal M_S|
    \leq
    N^{m+1}(N+1)^{2C_2(k+2)^{13B}}.
\]
Since
\[
    m=\left\lceil(k+2)^{15B}\right\rceil
\]
and $\log N\lesssim k+2$, it follows that
\begin{equation}\label{eq: num of matrices bound}
    \log|\mathcal M_S|
    \lesssim
    (k+2)^{16B}
\end{equation}
for all sufficiently large $k$.

Corollary~\ref{cor:sampling-signs} now gives an event
$\mathcal E_S$, depending only on the fixed support $S$, such that
\[
    \mathbb P_\varepsilon(\mathcal E_S)
    \geq
    1-\exp\bigl(-c_B(k+2)^{16B}\bigr)
    \geq
    1-N^{-100}
\]
for all sufficiently large $k$, and on $\mathcal E_S$,
\[
    |\varepsilon^{\mathsf T}M\varepsilon|
    \leq1
\]
simultaneously for every good matrix $M\in\mathcal M_S$.

\subsection{Deterministic matrix estimates and conclusion}
\label{sec: conclusion}

It remains to verify the deterministic hypotheses required in the
sampling argument.

For $b\in A$, write
\[
    Q_A(b):=\mathbb E_cH_{b,c}.
\]

\begin{lemma}\label{lem: frobenius bounds}
For every nonempty $A\subseteq G$,
\[
    \mathbb E_b\|Q_A(b)\|_F^2
    \leq
    \frac Da,
    \qquad
    \mathbb E_{b,c}\|H_{b,c}\|_F^2
    \leq
    D.
\]
Moreover,
\[
    \|\overline H_A\|_F^2
    \leq
    \frac Da
\]
and
\[
    \|M_A\|_F^2
    \leq
    \frac{Da}{W^2}.
\]
\end{lemma}

\begin{proof}
Fix $b\in A$. The relation
\[
    u-v=b-c
\]
determines $c$ uniquely from $u$ and $v$. Thus every nonzero entry of
$\mathbb E_cL_{b,c}$ is equal to $1/a$, and the number of such entries
is at most
\[
    \sum_{\substack{c\in A\\c\neq b}}
        r_S(b-c)
    \leq
    Da.
\]
Since symmetrization contracts the Frobenius norm,
\[
    \|Q_A(b)\|_F^2
    =
    \left\|
        \frac12\mathbb E_c
        (L_{b,c}+L_{b,c}^{\mathsf T})
    \right\|_F^2
    \leq
    \|\mathbb E_cL_{b,c}\|_F^2
    \leq
    \frac Da.
\]
Averaging in $b$ proves
\[
    \mathbb E_b\|Q_A(b)\|_F^2\leq\frac Da.
\]
Also,
\[
    \|\overline H_A\|_F^2
    =
    \|\mathbb E_bQ_A(b)\|_F^2
    \leq
    \mathbb E_b\|Q_A(b)\|_F^2
    \leq
    \frac Da.
\]
Since $M_A=(a/W)\overline H_A$, it follows that
\[
    \|M_A\|_F^2
    \leq
    \frac{Da}{W^2}.
\]

Finally, when $b\neq c$, the matrix $L_{b,c}$ has at most $D$
nonzero entries, while $L_{b,b}=0$. Hence
\[
    \|H_{b,c}\|_F^2
    \leq
    \|L_{b,c}\|_F^2
    \leq D,
\]
which gives
\[
    \mathbb E_{b,c}\|H_{b,c}\|_F^2\leq D.
\]
\end{proof}

\begin{lemma}\label{lem: operator norm bounds}
For every $b,c\in A$,
\[
    \|H_{b,c}\|_{\mathrm{op}}\leq1,
\]
and for every fixed $b\in A$,
\[
    \mathbb E_cH_{b,c}^2
    \preceq
    (k+2)^{-6B}I.
\]
\end{lemma}

\begin{proof}
The collision relation determines each row index from its column index
and conversely, so $L_{b,c}$ is a partial permutation matrix. Thus
\[
    \|L_{b,c}\|_{\mathrm{op}}\leq1,
\]
and consequently
\[
    \|H_{b,c}\|_{\mathrm{op}}\leq1.
\]

Both $L_{b,c}L_{b,c}^{\mathsf T}$ and
$L_{b,c}^{\mathsf T}L_{b,c}$ are diagonal. Moreover, fixing $u\in S$, the
$u$-th diagonal entry of $L_{b,c}L_{b,c}^{\mathsf T}$ can be nonzero
only when
\[
    c\in(b-u)+S.
\]
If $b-u\in J_+$, then the trimming factor makes this entry zero.
Otherwise,
\[
    n_+(b-u)\leq\tau,
\]
which implies that 
\[
    \mathbb E_c
    (L_{b,c}L_{b,c}^{\mathsf T})(u,u)
    \leq
    \frac{\tau}{a}.
\]
Similarly,
\[
    \mathbb E_c
    (L_{b,c}^{\mathsf T}L_{b,c})(v,v)
    \leq
    \frac{\tau}{a}
\]
for every $v\in S$. Hence
\[
    \mathbb E_cL_{b,c}L_{b,c}^{\mathsf T}
    \preceq
    \frac{\tau}{a}I,
    \qquad
    \mathbb E_cL_{b,c}^{\mathsf T}L_{b,c}
    \preceq
    \frac{\tau}{a}I.
\]

Finally,
\[
    \frac12
    (LL^{\mathsf T}+L^{\mathsf T}L)
    -
    \left(
        \frac{L+L^{\mathsf T}}2
    \right)^2
    =
    \frac14
    (L-L^{\mathsf T})^{\mathsf T}
    (L-L^{\mathsf T})
    \succeq0.
\]
Applying this identity with $L=L_{b,c}$ and averaging in $c$ gives
\[
    \mathbb E_cH_{b,c}^2
    \preceq
    \frac{\tau}{a}I
    =
    (k+2)^{-6B}I.
\]
\end{proof}

We can now apply the sampling results from
Section~\ref{sec:sampling}. Since $A$ is admissible,
\[
    a > \frac{W}{(k+2)^2},
\]
and hence
\[
    \frac Da
    \lesssim
    \frac{(k+2)^4}{W}
    \leq
    \frac{(k+2)^{100}}{W}.
\]
Moreover,
\[
    \mathbb E_{b,c}\|H_{b,c}\|_F^2
    \leq
    D
    \lesssim
    (k+2)^2
    \leq
    (k+2)^{100},
\]
while the cardinality assumption on $A$ yields
\[
    \lambda_A
    =
    \frac aW
    \lesssim
    (k+2)^{100}.
\]
Together with Lemma~\ref{lem: operator norm bounds}, these are precisely
the hypotheses of Lemma~\ref{lem:sampling}.

It remains only to verify the crude bound required in
Corollary~\ref{cor:sampling-transfer}. The same entry count used above
yields
\[
    |\varepsilon^{\mathsf T}H_{b,c}\varepsilon|
    =
    |\varepsilon^{\mathsf T}L_{b,c}\varepsilon|
    \leq D.
\]
Therefore, uniformly in the auxiliary sample and in the signs,
\[
    |\varepsilon^{\mathsf T}\widehat M_A\varepsilon|
    \leq
    D\lambda_A
    \lesssim
    (k+2)^{102}.
\]
Thus Corollary~\ref{cor:sampling-transfer} applies, and 
on the event $\mathcal E_S$ we consequently have
\[
    \sup_{A\ {\rm admissible}}
    |\varepsilon^{\mathsf T}M_A\varepsilon|
    \leq2.
\]

Combining this estimate with
\eqref{eq:trimmed-conv-expansion} and recalling the small-set estimate
\eqref{eq: small sets contrib}, we conclude that
\[
    \frac{\|Z\ast\mathbf1_A\|_{2}^2}
         {W|A|}
    \lesssim1
\]
simultaneously for every
\[
    A\subseteq G,
    \qquad
    1\leq|A|\leq C_1W(k+2)^{100}.
\]
Since
\[
    \mathbb P_\varepsilon(\mathcal E_S)
    \geq1-N^{-100},
\]
this proves Proposition~\ref{prop: fixed support prop}, and with it, Theorem \ref{t:main} and Theorem \ref{thm: main theorem}.

\appendix

\section{A Sketch of an Earlier Argument}
\label{app:old-argument}

In this appendix we briefly describe an earlier argument of the second
author which yields the estimate in Proposition~\ref{p:key} with an additional factor
$\log(k+2)$ on the right-hand side; the below sketch highlights the different mechanisms by which the old and new arguments solve the same uniformity problem coming from the presence of a supremum over a very large class of sets. 

At its core, the proof is based on
Fourier analysis and metric entropy, with the key gain coming from the
diffuse nature of the Fourier transform of a sparse random set. In what
follows, we only sketch the main ideas and deliberately omit many of
the technical details.

Fixing a scale $k$, we pass as before to a cyclic group
\[
    G=\mathbb Z/N\mathbb Z,
    \qquad
    N\approx2^k,
    \qquad
    W:=W_k\approx N^{1/2},
\]
and write
\[
    Y(n):=(X_n-\sigma_n)\mathbf 1_{I_k}(n),
\]
identifying $I_k$ with its image inside $G.$

For $A\subseteq G$, we let
\[
    b_A
    :=
    \mathbf 1_A-\frac{|A|}{N}\mathbf 1_G
\]
denote the mean-zero version of $\mathbf 1_A$. We introduce the random
supremum
\[
    \mathcal Q_k(Y)
    :=
    \sup_{\substack{A\subseteq G\\
                    1\leq |A|\leq C_1W(k+2)^{100}}}
    \frac{\|Y*b_A\|_2^2}{W|A|};
\]
the omitted zero-frequency contribution is handled directly by way of Bernstein's inequality.

Rather than estimating $\mathcal Q_k(Y)$ almost surely directly, we reduce to a moment estimate.
Specifically, setting
\[
    q = q_k:=\log(k+2),
\]
an application of Markov's inequality and the Borel--Cantelli lemma
shows that the desired almost-sure logarithmic estimate follows once one
proves
\[
    \|\mathcal Q_k(Y)\|_{L^{q}}
    \lesssim
    \log(k+2).
\]

To prove this moment estimate, we first expand $\|Y*b_A\|_2^2$. The
diagonal term can be estimated directly, while the off-diagonal term is
a quadratic form in the independent random variables defining $Y$. A
probabilistic decoupling inequality in the spirit of \cite[Theorem 3.1.1]{deLaPenaGine1999Decoupling} allows us to introduce an
independent copy $Y'$ of $Y$ in the analysis and reduces the problem, up to bounded
terms, to controlling the relevant moment of the random variable
\[
    \sup_A
    \frac{
        |\langle Y*b_A,Y'*b_A\rangle|
    }{W|A|}.
\]

The next step is Fourier analytic. Let
\[
    \Omega=\Omega(Y)\subseteq G\setminus\{0\}
\]
consist of a suitably chosen number of frequencies at which
$|\widehat Y|$ is largest; in the argument, one takes
\[
    |\Omega|
    =
    \Big\lfloor
        \frac{N}{(k+2)^{B_1}}
    \Big\rfloor,
\]
where $B_1>0$ is a sufficiently large fixed constant. Using
Plancherel's identity, we split the preceding inner product into its
contributions on $\Omega$ and on its complement.

Away from $\Omega$, the Fourier coefficients of $Y$ are uniformly
smaller than the exceptional ones; an estimate based on
Bernstein's inequality yields the required logarithmic moment bound for
this contribution.

It remains to control the contribution from $\Omega$. The important
point is that $\Omega$ depends only on $Y$ and is therefore independent
of the copy $Y'$. Thus, from the point of view of $Y'$, we are reduced
to a restricted operator estimate for an arbitrary fixed set of
relatively few frequencies.

More precisely, we define the set 
\[
    \mathcal U_{\Omega}
    :=
    \Big\{
        \frac{P_\Omega b_A}{\sqrt{|A|}}:
        1\leq |A|\leq C_1W(k+2)^{100}
    \Big\} \subseteq \ell^2(G), 
\]
where $P_\Omega$ denotes the Fourier projection onto $\Omega$, and, for
$y\in\ell^2(G)$, we set
\[
    F_\Omega(y)
    :=
    \sup_{u\in\mathcal U_{\Omega}}
    \frac1W\|y*u\|_2^2.
\]
Our aim is therefore to estimate a suitable moment of
$F_\Omega(Y')$, uniformly over $\Omega$ of the prescribed cardinality.

Introducing another independent copy $Y''$ and using the convexity of $F_\Omega$,
one has
\[
    \|F_\Omega(Y')\|_{L^q}
    \leq
    \|F_\Omega(Y''-Y')\|_{L^q}.
\]
Writing
\[
    Y''-Y'
    =
    \sum_{n\in S}\varepsilon_n\delta_n,
\]
we may condition on the random support $S$, so that the
remaining randomness is carried by the independent Rademacher signs. 

The main gain in the restricted estimate comes from the interaction
between the sparsity of $S$ and the fact that $\Omega$ occupies only a
small portion of frequency space. At a qualitative level, the uncertainty principle suggests that a function which is highly localized in physical space cannot also be highly localized in frequency space; thus, the sparsity of \(S\) should force its Fourier transform to be spread over many frequencies. Randomness strengthens this heuristic by producing cancellation at nonzero frequencies and preventing the Fourier mass from concentrating on a prescribed small set.

This heuristic can already be seen in the simpler model in which each element of \(G\) is included in \(S\) independently with probability \(0<p\ll1\), so that \(p\) controls the expected sparsity of the resulting random set: randomness prevents concentration on any set of distinguished frequencies, while sparsity makes the resulting interaction with a fixed frequency set quantitatively small. Although the random set arising in our argument is generated with nonconstant inclusion probabilities, the same heuristic remains valid.

The restricted operator arising in the proof measures precisely this interaction between the random spatial set $S$ and the set of frequencies $\Omega$. Indeed, the relevant quantity is governed by
\[
R_{\Omega,S}=R_\Omega UJ_S,
\]
where $J_S: \ell^2(S) \hookrightarrow \ell^2(G)$ is coordinate inclusion, $U$ is the unitary Fourier transform on $\ell^2(G)$ and $R_\Omega = J_\Omega^*$ is coordinate projection onto $\Omega$.
The preceding heuristic suggests that the operator norm $\|R_{\Omega,S}\|_{2\to2}$ should be small: because $S$ is sparse and Fourier-diffuse, a vector supported on $S$ should not be able to place too much of its Fourier energy inside a fixed low-dimensional collection of frequencies on average. Matrix concentration estimates make this statement quantitative.

At this point,
the main remaining difficulty is to retain this gain uniformly over the very
large family of admissible sets. This is where metric entropy enters:
%On $\Omega$, the Fourier transform of an admissible vector has the form
%\[
%    \frac1{\sqrt{|A|}}
%    \sum_{x\in A}\psi_x,
%\]
%where the vectors $\psi_x$ belong to a fixed collection of $N$ phase
%vectors. Thus,
after normalization, the relevant Fourier vectors lie
in a suitable dilation of the convex hull of a fixed collection of $N$ phase vectors; although the original family of sets $A$ is combinatorially enormous, its image in the relevant Fourier metric has substantially smaller
complexity.

To turn this metric-entropy information into a bound for the
supremum, we apply the quadratic-process estimate of Krahmer, Mendelson
and Rauhut \cite{KrahmerMendelsonRauhut2014}. Conditional on the random
support $S$, this estimate bounds the relevant moments in terms of three
quantities: the maximal Frobenius norm of the associated convolution
matrices, their maximal operator norm, and the corresponding
$\gamma_2$-functional in the operator-norm metric.

These three quantities reflect the three basic features of the
argument. The Frobenius norm is controlled by the sparsity of $S$;
the operator norm is controlled by the Fourier-diffuse nature of $S$, addressed by the matrix-concentration estimate for
\[
    R_{\Omega,S}=R_\Omega UJ_S;
\]
the $\gamma_2$-functional is controlled by exploiting the convex-hull
structure of the projected Fourier vectors. Combining these estimates,
and then averaging over the randomness of $S$, yields
\[
    \|F_\Omega(Y')\|_{L^{q_k}}
    \lesssim1
\]
uniformly over the relevant fixed frequency sets $\Omega$.

Combining
this restricted estimate with the estimate outside the exceptional
Fourier frequencies yields
\[
    \|\mathcal Q_k(Y)\|_{L^{q_k}}
    \lesssim
    \log(k+2),
\]
and hence
\[
    \|\nu_k*\mathbf 1_A\|_2^2
    \lesssim_{\omega,C_1}
    \frac{\log(k+2)}{W_k}|A|
\]
almost surely, for all $k$.

The argument above and the sampling proof in the main body of this paper
address the same basic uniformity problem in rather different ways. In
the earlier proof, the gain comes from being diffuse in Fourier space, while metric entropy allows that gain to be retained uniformly over the large class
of admissible sets; the price of this entropy argument is an additional factor $\log(k+2).$ On the other hand, the present proof avoids this moment and metric-entropy mechanism.
Instead, each quadratic form $M_A$ is replaced by an empirical
representative $\widehat M_A$, which depends on far fewer parameters. The collection of all possible sampled
matrices has sufficiently small cardinality that it can be controlled
directly by Hanson--Wright and a union bound, while the identity
\[
    \mathbb E_{\mathrm{aux}}\widehat M_A=M_A
\]
allows us to transfer the resulting estimate back to the original family. In this
way, the sampling argument provides a direct finite compression of the
family of quadratic forms and removes the remaining logarithmic loss.

% Embedded bibliography: no external .bib file or BibTeX run is required.

\end{document}